\documentclass[11pt]{amsart}%
\usepackage{amsfonts}
\usepackage{amsmath}
\usepackage{amssymb,latexsym}
\usepackage{graphicx}
\usepackage{titletoc}
\usepackage{amssymb}
\usepackage{color}%
\providecommand{\U}[1]{\protect \rule{.1in}{.1in}}

\newtheorem{theorem}{Theorem}[section]
\theoremstyle{plain}

\newtheorem{corollary}{Corollary}[section]

\newtheorem{lemma}{Lemma}[section]

\newtheorem{remark}{Remark}[section]

\numberwithin{equation}{section}
\begin{document}
\title[Hamiltonian system for the elliptic form of Painlev\'{e} VI]{Hamiltonian system for the elliptic form of Painlev\'{e} VI equation}
\author{Zhijie Chen}
\address{Center for Advanced Study in Theoretical Sciences (CASTS), National Taiwan
University, Taipei 10617, Taiwan }
\email{chenzhijie1987@sina.com}
\author{Ting-Jung Kuo}
\address{Taida Institute for Mathematical Sciences (TIMS), National Taiwan University,
Taipei 10617, Taiwan }
\email{tjkuo1215@gmail.com}
\author{Chang-Shou Lin}
\address{Taida Institute for Mathematical Sciences (TIMS), Center for Advanced Study in
Theoretical Sciences (CASTS), National Taiwan University, Taipei 10617, Taiwan }
\email{cslin@math.ntu.edu.tw}

\begin{abstract}
In literature, it is known that any solution of Painlev\'{e} VI equation
governs the isomonodromic deformation of a second order linear Fuchsian ODE on
$\mathbb{CP}^{1}$. In this paper, we extend this isomonodromy theory on
$\mathbb{CP}^{1}$ to the moduli space of elliptic curves by studying the
isomonodromic deformation of the generalized Lam\'{e} equation. Among other
things, we prove that the isomonodromic equation is a new Hamiltonian system,
which is equivalent to the elliptic form of Painlev\'{e} VI equation for
generic parameters. For Painlev\'{e} VI equation with some special parameters,
the isomonodromy theory of the generalized Lam\'{e} equation greatly
simplifies the computation of the monodromy group in $\mathbb{CP}^{1}$. This
is one of the advantages of the elliptic form.

\end{abstract}
\maketitle

\section{Introduction}

The isomonodromic deformation plays an universal role to connect many
different research areas of mathematics and physics. Our purpose of this paper
is to develop an isomonodromy theory for the generalized Lam\'{e} equation on
the moduli space of elliptic curves.

\subsection{Painlev\'{e} VI in elliptic form}

Historically, the discovery of Painlev\'{e} equations was originated from the
research on complex ODEs from the middle of 19th century up to early 20th
century, led by many famous mathematicians including Painlev\'{e} and his
school. The aim is to classify those nonlinear ODEs whose solutions have the
so-called \textit{Painlev\'{e} property}. We refer the reader to
\cite{Babich-Bordag,DIKZ,Dubrovin-Mazzocco,Guzzetti,Halphen,Hit1,Hit2,Inaba-Iwasaki-Saito,GP,Kawai,Lisovyy-Tykhyy,Y.Manin,Okamoto2,Okamoto1,Okamoto,Poole,Sasaki}
and references therein for some historic account and the recent developments.
Painlev\'{e} VI with four free parameters $(\alpha,\beta,\gamma,\delta)$ can
be written as%
\begin{align}
\frac{d^{2}\lambda}{dt^{2}}= &  \frac{1}{2}\left(  \frac{1}{\lambda}+\frac
{1}{\lambda-1}+\frac{1}{\lambda-t}\right)  \left(  \frac{d\lambda}{dt}\right)
^{2}-\left(  \frac{1}{t}+\frac{1}{t-1}+\frac{1}{\lambda-t}\right)
\frac{d\lambda}{dt}\label{46}\\
&  +\frac{\lambda \left(  \lambda-1\right)  \left(  \lambda-t\right)  }%
{t^{2}\left(  t-1\right)  ^{2}}\left[  \alpha+\beta \frac{t}{\lambda^{2}%
}+\gamma \frac{t-1}{\left(  \lambda-1\right)  ^{2}}+\delta \frac{t\left(
t-1\right)  }{\left(  \lambda-t\right)  ^{2}}\right]  .\nonumber
\end{align}
In the literature, it is well-known that Painlev\'{e} VI (\ref{46}) is closely
related to the isomonodromic deformation of either a $2\times2$ linear ODE
system of first order (under the non-resonant condition, the isomonodromic
equation is known as the Schlesinger system; see \cite{Jimbo-Miwa}) or a
second order Fuchsian ODE (under the non-resonant condition, the isomonodromic
equation is a Hamiltonian system; see \cite{Fuchs,Okamoto2}). This associated
second order Fuchsian ODE is defined on $\mathbb{CP}^{1}$ and has five regular
singular points $0,1,t,\lambda(t)$ and $\infty$. Among them, $\lambda(t)$ (as
a solution of Painlev\'{e} VI) is an apparent singularity. This isomonodromy
theory on $\mathbb{CP}^{1}$ was first discovered by R. Fuchs \cite{Fuchs}, and
later generalized to the $n$-dimensional Garnier system by K. Okamoto
\cite{Okamoto2}. We will briefly review this classical isomonodromy theory in
Section 4.

Throughout the paper, we use the notations $\omega_{0}=0,\omega_{1}%
=1,\omega_{2}=\tau$, $\omega_{3}=1+\tau$, $\Lambda_{\tau}=\mathbb{Z+\tau Z}$,
and $E_{\tau}\doteqdot \mathbb{C}/\Lambda_{\tau}$ where $\tau \in \mathbb{H}%
=\left \{  \tau|\operatorname{Im}\tau>0\right \}  $ (the upper half plane). We
also define $E_{\tau}\left[  2\right]  \doteqdot \left \{  \frac{\omega_{i}}%
{2}|i=0,1,2,3\right \}  $ to be the set of 2-torsion points in the flat torus
$E_{\tau}$. From the Painlev\'{e} property of (\ref{46}), any solution
$\lambda \left(  t\right)  $ is a multi-valued meromorphic function in
$\mathbb{C}\backslash \left \{  0,1\right \}  $. To avoid the multi-valueness of
$\lambda \left(  t\right)  $, it is better to lift solutions of (\ref{46}) to
its universal covering. It is known that the universal covering of
$\mathbb{C}\backslash \left \{  0,1\right \}  $ is $\mathbb{H}$. Then $t$ and the
solution $\lambda \left(  t\right)  $ can be lifted to $\tau$ and $p\left(
\tau \right)  $ respectively through the covering map by%
\begin{equation}
t\left(  \tau \right)  =\frac{e_{3}(\tau)-e_{1}(\tau)}{e_{2}(\tau)-e_{1}(\tau
)}\text{ and }\lambda(t)=\frac{\wp(p(\tau)|\tau)-e_{1}(\tau)}{e_{2}%
(\tau)-e_{1}(\tau)}, \label{tr}%
\end{equation}
where $\wp \left(  z|\tau \right)  $ is the Weierstrass elliptic function
defined by%
\[
\wp \left(  z|\tau \right)  =\frac{1}{z^{2}}+\sum_{\omega \in \Lambda_{\tau
}\backslash \left \{  0\right \}  }\left[  \frac{1}{\left(  z-\omega \right)
^{2}}-\frac{1}{\omega^{2}}\right]  ,
\]
and $e_{i}=\wp \left(  \frac{\omega_{i}}{2}|\tau \right)  $, $i=1,2,3$. Then
$p\left(  \tau \right)  $ satisfies the following elliptic form%
\begin{equation}
\frac{d^{2}p\left(  \tau \right)  }{d\tau^{2}}=\frac{-1}{4\pi^{2}}\sum
_{i=0}^{3}\alpha_{i}\wp^{\prime}\left(  p\left(  \tau \right)  +\frac
{\omega_{i}}{2}|\tau \right)  , \label{124}%
\end{equation}
where $\wp^{\prime}\left(  z|\tau \right)  =\frac{d}{dz}\wp \left(
z|\tau \right)  $ and
\begin{equation}
\left(  \alpha_{0},\alpha_{1},\alpha_{2},\alpha_{3}\right)  =\left(
\alpha,-\beta,\gamma,\frac{1}{2}-\delta \right)  . \label{126}%
\end{equation}
This elliptic form was already known to Painlev\'{e} \cite{Painleve}. For a
modern proof, see \cite{Babich-Bordag, Y.Manin}.

The advantage of (\ref{124}) is that $\wp(p(\tau)|\tau)$ is single-valued for
$\tau \in \mathbb{H}$, although $p(\tau)$ has a branch point at those $\tau_{0}$
such that $p(\tau_{0})\in E_{\tau_{0}}[2]$ (see e.g. (\ref{asymp}) below). We
take $(\alpha_{0},\alpha_{1},\alpha_{2},\alpha_{3})=(\frac{1}{8},\frac{1}%
{8},\frac{1}{8},\frac{1}{8})$ for an example to explain it. Painlev\'{e} VI
with this special parameter has connections with some geometric problems; see
\cite{Chen-Kuo-Lin, Hit1}. In the seminal work \cite{Hit1}, N. Hitchin
discovered that, for a pair of complex numbers $(r,s)\in \mathbb{C}%
^{2}\backslash \frac{1}{2}\mathbb{Z}^{2}$, $p(\tau)$ defined by the following
formula:%
\begin{equation}
\wp \left(  p(\tau)|\tau \right)  =\wp \left(  r+s\tau|\tau \right)  +\frac
{\wp^{\prime}\left(  r+s\tau|\tau \right)  }{2\left(  \zeta \left(  r+s\tau
|\tau \right)  -r\eta_{1}(\tau)-s\eta_{2}(\tau)\right)  }, \label{513}%
\end{equation}
is a solution to (\ref{124}) with $\alpha_{k}=\frac{1}{8}$ for all $k$. Here
$\zeta \left(  z|\tau \right)  \doteqdot-\int^{z}\wp(\xi|\tau)d\xi$ is the
Weierstrass zeta function and has the quasi-periods%
\begin{equation}
\zeta \left(  z+1|\tau \right)  =\zeta \left(  z|\tau \right)  +\eta_{1}%
(\tau)\text{ and }\zeta \left(  z+\tau|\tau \right)  =\zeta \left(
z|\tau \right)  +\eta_{2}(\tau). \label{40-2}%
\end{equation}
{By (\ref{513}), Hitchin could construct an Einstein metric with positive
curvature if }$r\in \mathbb{R}$ and $s\in i\mathbb{R}$, and an Einstein metric
with negative curvature if $r\in i\mathbb{R}$ and $s\in \mathbb{R}$. It follows
from (\ref{513}) that $\wp \left(  p(\tau)|\tau \right)  $ is a single-valued
meromorphic function in $\mathbb{H}$. However, each $\tau_{0}$ with
$p(\tau_{0})\in E_{\tau_{0}}[2]$ is a branch point of order $2$ for $p(\tau)$.

Motivated from Hitchin's solutions, we would like to extend the beautiful
formula (\ref{513}) to Painlev\'{e} VI with other parameters. But it is not a
simple matter because it invloves complicated derivatives with respect to the
moduli parameter $\tau$. For example, for Hichin's solutions, it seems not
easy to derive (\ref{124}) with $\alpha_{k}=\frac{1}{8}$ for all $k$ directly
from the formula (\ref{513}). We want to provide a systematical way to study
this problem. To this goal, the first step is to develop a theory in the
moduli space of tori which is analogous to the Fuchs-Okamoto theory on
$\mathbb{CP}^{1}$. The purpose of this paper is to derive the Hamiltonian
system for the elliptic form (\ref{124}) by developing such an isomonodromy
theory in the moduli space of tori. The key issue is \textit{what the linear
Fuchsian equation in tori is such that its isomonodromic deformation is
related to the elliptic form} (\ref{124}).

\subsection{Generalized Lam\'{e} equation}

Motivated from our study of the surprising connection of the mean field
equation and the elliptic form (\ref{124}) of Painlev\'{e} VI in
\cite{Chen-Kuo-Lin}, our choice of the Fuchsian equation is the generalized
Lam\'{e} equation defined by (\ref{505}) below. More precisely, let us
consider the following mean field equation%
\begin{equation}
\Delta u+e^{u}=8\pi \sum_{i=0}^{3}n_{k}\delta_{\frac{\omega_{k}}{2}}%
+4\pi \left(  \delta_{p}+\delta_{-p}\right)  \text{ in }E_{\tau},\label{501}%
\end{equation}
where $n_{k}>-1$, $\delta_{\pm p}$ and $\delta_{\frac{\omega_{k}}{2}}$ are the
Dirac measure at $\pm p$ and $\frac{\omega_{k}}{2}$ respectively. By the
Liouville theorem, any solution $u$ to equation (\ref{501}) could be written
into the following form:%
\begin{equation}
u(z)=\log \frac{8|f^{\prime}\left(  z\right)  |^{2}}{(1+|f\left(  z\right)
|^{2})^{2}},\label{502}%
\end{equation}
where $f\left(  z\right)  $ is a meromorphic function in $\mathbb{C}$.
Conventionally $f\left(  z\right)  $ is called a developing map of $u$. We
could see below that there associates a 2nd order complex ODE reducing from
the nonlinear PDE (\ref{501}). Indeed, it follows from (\ref{501}) that
outside $E_{\tau}\left[  2\right]  \cup \left \{  \pm p\right \}  $,
\begin{align*}
\left(  u_{zz}-\frac{1}{2}u_{z}^{2}\right)  _{\bar{z}}  & =\left(  u_{z\bar
{z}}\right)  _{z}-u_{z}u_{z\bar{z}}\\
& =\left(  -\frac{1}{4}e^{u}\right)  _{z}+\frac{1}{4}e^{u}u_{z}=0.
\end{align*}
So $u_{zz}-\frac{1}{2}u_{z}^{2}$ is an elliptic function on the torus
$E_{\tau}$ with singularties at $E_{\tau}\left[  2\right]  \cup \left \{  \pm
p\right \}  $. Since the behavior of $u$ is fixed by the RHS of (\ref{501}),
for example, $u(z)=2\log|z-p|+O(1)$ near $p$, we could compute explicitly the
dominate term of $u_{zz}-\frac{1}{2}u_{z}^{2}$ near each singular point. Let
us further assume that $u(z)$ is \textit{even}, i.e., $u(z)=u(-z)$. Then we
have%
\begin{align}
&  u_{zz}-\frac{1}{2}u_{z}^{2}\label{503}\\
= &  -2\left[
\begin{array}
[c]{l}%
\sum_{k=0}^{3}n_{k}\left(  n_{k}+1\right)  \wp \left(  z+\frac{\omega_{k}}%
{2}\right)  +\frac{3}{4}\left(  \wp \left(  z+p\right)  +\wp \left(  z-p\right)
\right)  \\
+A\left(  \zeta \left(  z+p\right)  -\zeta \left(  z-p\right)  \right)  +B
\end{array}
\right]  \nonumber \\
\doteqdot &  -2I\left(  z\right)  ,\nonumber
\end{align}
where $A,B$ are two (unknown) complex numbers.

On the other hand, we could deduce from (\ref{502}) that the Schwarzian
derivative $\{f;z\}$ of $f$ can be expressed by%
\begin{equation}
\{f;z\} \doteqdot \left(  \frac{f^{\prime \prime}}{f^{\prime}}\right)  ^{\prime
}-\frac{1}{2}\left(  \frac{f^{\prime \prime}}{f^{\prime}}\right)  ^{2}%
=u_{zz}-\frac{1}{2}u_{z}^{2}=-2I\left(  z\right)  .\label{504}%
\end{equation}
From (\ref{504}), we connect the developing maps of an even solution $u$ to
(\ref{501}) with the following generalized Lam\'{e} equation%
\begin{equation}
y^{\prime \prime}(z)=I(z)y(z)\text{ in }E_{\tau},\label{505}%
\end{equation}
where the potential $I(z)$ is given by (\ref{503}). In the classical
literature, the 2nd order ODE%
\begin{equation}
y^{\prime \prime}(z)=\left(  n(n+1)\wp(z)+B\right)  y(z)\text{ \ in \ }E_{\tau
}\label{503-1}%
\end{equation}
is called the Lam\'{e} equation, and has been extensively studied since the
19th century, particularly for the case $n\in \mathbb{Z}/2$. See
\cite{Chai-Lin-Wang,Halphen, Poole, Whittaker-Watson} and the references
therein. In this paper, we will prove that (\ref{503-1}) appears as a limiting
equation of (\ref{505}) under some circumstances (see Theorem
\ref{thm-II-9 copy(1)}).

From (\ref{502}) and (\ref{504}), any two developing maps $f_{i}$, $i=1,2$ of
the same solution $u$ must satisfy%
\[
f_{2}(z)=\alpha \cdot f_{1}(z)\doteqdot \frac{af_{1}(z)+b}{cf_{1}(z)+d}%
\]
for some $\alpha=\left(
\begin{matrix}
a & b\\
c & d
\end{matrix}
\right)  \in PSU(2)$. From here, we could define a projective monodromy
representation $\rho:\pi_{1}(E_{\tau}\backslash(E_{\tau}\left[  2\right]
\cup \{ \pm p\}),q_{0})\rightarrow PSU(2)$, where $q_{0}\not \in E_{\tau}\left[
2\right]  \cup \{ \pm p\}$ is a base point. Indeed, any developing map $f$ might
be multi-valued. For any loop $\ell \in \pi_{1}(E_{\tau}\backslash(E_{\tau
}\left[  2\right]  \cup \{ \pm p\}),q_{0})$, $\ell^{\ast}f$ denotes the analytic
continuation of $f$ along $\ell$. Since $\ell^{\ast}f$ is also a developing
map of the same $u$, there exists $\rho(\ell)\in PSU(2)$ such that $\ell
^{\ast}f=\rho(\ell)\cdot f$. Thus, the map $\ell \mapsto \rho(\ell)$ defines a
group homomorphism from $\pi_{1}(E_{\tau}\backslash(E_{\tau}\left[  2\right]
\cup \{ \pm p\}),q_{0})$ to $PSU(2)$. When $n_{k}\in \mathbb{N}\cup \left \{
0\right \}  $ for all $k$, the developing map is a single-valued meromorphic
function defined in $\mathbb{C}$, and $\left \{  \pm p\right \}  $ would become
apparent singularties. Thus, the projective monodromy representation would be
reduced to a homomorphism from $\pi_{1}\left(  E_{\tau},q_{0}\right)  $ to
$PSU(2)$. This could greatly simplify the computation of the monodromy group.
The deep connection of the mean field equation (\ref{501}) and the elliptic
form (\ref{124}) has been discussed in detail in \cite{Chen-Kuo-Lin}. This is
our motivation to study the isomonodromic deformation of the generalized
Lam\'{e} equation (\ref{505}) in this paper.

\subsection{Isomonodromic deformation and Hamiltonian system}

We are now in a position to state our main results. Recall the generalized
Lam\'{e} equation (\ref{505}):%
\begin{equation}
y^{\prime \prime}=\left[
\begin{array}
[c]{l}%
\sum_{k=0}^{3}n_{k}\left(  n_{k}+1\right)  \wp \left(  z+\frac{\omega_{k}}%
{2}\right)  +\frac{3}{4}\left(  \wp \left(  z+p\right)  +\wp \left(  z-p\right)
\right) \\
+A\left(  \zeta \left(  z+p\right)  -\zeta \left(  z-p\right)  \right)  +B
\end{array}
\right]  y. \label{89-0}%
\end{equation}
Equation (\ref{89-0}) has no solutions with logarithmic singularity at
$\frac{\omega_{k}}{2}$ unless $n_{k}\in \frac{1}{2}+\mathbb{Z}$. Therefore, we
need to assume the non-resonant condition: $n_{k}\not \in \frac{1}%
{2}+\mathbb{Z}$ for all $k$. Observe that the exponent difference of
(\ref{89-0}) at $\pm p$ is $2$. Here the singular points $\pm p$ are always
assumed to be \textit{apparent}. Under this assumption, the coefficients $A$
and $B$ together satisfy (\ref{101-0}) below. Our first main result is following.

\begin{theorem}
\label{theorem1-2} Let
\begin{equation}
\alpha_{k}=\frac{1}{2}\left(  n_{k}+\frac{1}{2}\right)  ^{2}\text{ with }%
n_{k}\not \in \frac{1}{2}+\mathbb{Z},\ k=0,1,2,3. \label{125}%
\end{equation}
Then $p\left(  \tau \right)  $ is a solution of the elliptic form (\ref{124})
if and only if there exist $A\left(  \tau \right)  $ and $B\left(  \tau \right)
$ such that the generalized Lam\'{e} equation (\ref{89-0}) with apparent
singularities at $\pm p\left(  \tau \right)  $ preserves the monodromy while
$\tau$ is deforming.
\end{theorem}

Our method to prove Theorem \ref{theorem1-2} consists of two steps: the first
is to derive the isomonodromic equation, a Hamiltonian system, in the moduli
space of tori for (\ref{89-0}) under the non-resonant condition $n_{k}%
\not \in \frac{1}{2}+\mathbb{Z}$. The second is to prove that this
isomonodromic equation (the Hamiltonian system) is equivalent to the elliptic
form (\ref{124}). To describe the isomonodromic equation for (\ref{89-0}), we
let the Hamiltonian $K\left(  p,A,\tau \right)  $ be defined by%
\begin{align}
K\left(  p,A,\tau \right)   &  =\frac{-i}{4\pi}(B+2p\eta_{1}\left(
\tau \right)  A)\label{143-0}\\
&  =\frac{-i}{4\pi}\left(
\begin{array}
[c]{l}%
A^{2}+\left(  -\zeta \left(  2p|\tau \right)  +2p\eta_{1}\left(  \tau \right)
\right)  A-\frac{3}{4}\wp \left(  2p|\tau \right)  \\
-\sum_{k=0}^{3}n_{k}\left(  n_{k}+1\right)  \wp \left(  p+\frac{\omega_{k}}%
{2}|\tau \right)
\end{array}
\right)  .\nonumber
\end{align}
Consider the Hamiltonian system%
\begin{equation}
\left \{
\begin{array}
[c]{l}%
\frac{dp\left(  \tau \right)  }{d\tau}=\frac{\partial K\left(  p,A,\tau \right)
}{\partial A}=\frac{-i}{4\pi}\left(  2A-\zeta \left(  2p|\tau \right)
+2p\eta_{1}\left(  \tau \right)  \right)  \\
\frac{dA\left(  \tau \right)  }{d\tau}=-\frac{\partial K\left(  p,A,\tau
\right)  }{\partial p}=\frac{i}{4\pi}\left(
\begin{array}
[c]{l}%
\left(  2\wp \left(  2p|\tau \right)  +2\eta_{1}\left(  \tau \right)  \right)
A-\frac{3}{2}\wp^{\prime}\left(  2p|\tau \right)  \\
-\sum_{k=0}^{3}n_{k}\left(  n_{k}+1\right)  \wp^{\prime}\left(  p+\frac
{\omega_{k}}{2}|\tau \right)
\end{array}
\right)
\end{array}
\right.  .\label{142-0}%
\end{equation}
Then our first step leads to the following result:

\begin{theorem}
[=Theorem \ref{thm4-2}]\label{theorem1-3}Let $n_{k}\not \in \frac{1}%
{2}+\mathbb{Z}$, $k=0,1,2,3$. Then $(p(\tau),A(\tau))$ satisfies the
Hamiltonian system (\ref{142-0}) if and only if equation (\ref{89-0}) with
$(p(\tau),A(\tau),B(\tau))$ preserves the monodromy, where
\begin{equation}
B=A^{2}-\zeta \left(  2p\right)  A-\frac{3}{4}\wp \left(  2p\right)  -\sum
_{k=0}^{3}n_{k}\left(  n_{k}+1\right)  \wp \left(  p+\frac{\omega_{k}}%
{2}\right)  .\label{101-0}%
\end{equation}

\end{theorem}

And the second step is to prove

\begin{theorem}
[=Theorem \ref{thm4-1}]\label{theorem1-4}The elliptic form (\ref{124}) is
equivalent to the Hamiltonian system (\ref{142-0}), where $\alpha_{k}=\frac
{1}{2}\left(  n_{k}+\frac{1}{2}\right)  ^{2},$\ $k=0,1,2,3$.
\end{theorem}

Clearly Theorem \ref{theorem1-2} follows from Theorems \ref{theorem1-3} and
\ref{theorem1-4} directly.

\begin{remark}
In \cite{Y.Manin}, Manin rewrote the elliptic form (\ref{124}) into an obvious
time-dependent Hamiltonian system:%
\begin{equation}
\frac{dp\left(  \tau \right)  }{d\tau}=\frac{\partial H}{\partial q},\text{
}\frac{dq\left(  \tau \right)  }{d\tau}=-\frac{\partial H}{\partial p},
\label{HM}%
\end{equation}
where%
\[
H=H\left(  \tau,p,q\right)  \doteqdot \frac{q^{2}}{2}+\frac{1}{4\pi^{2}}%
\sum_{i=0}^{3}\alpha_{i}\wp \left(  p\left(  \tau \right)  +\frac{\omega_{i}}%
{2}|\tau \right)  .
\]
However, it is not clear whether the Hamiltonian system (\ref{HM}) governs
isomonodromic deformations of any Fuchsian equations in $E_{\tau}$ or not.
Different from (\ref{HM}), our Hamiltonian system (\ref{142-0}) governs
isomonodromic deformations of the generalized Lam\'{e} equation for generic parameters.
\end{remark}

Both Theorems \ref{theorem1-3} and \ref{theorem1-4} are proved in Section 2.
It seems that the generalized Lam\'{e} equation (\ref{89-0}) looks simpler
than the corresponding Fuchsian ODE on $\mathbb{CP}^{1}$, and it is the same
for the Hamiltonian system (\ref{142-0}), compared to the corresponding one on
$\mathbb{CP}^{1}$. From the second equation of (\ref{142-0}), $A\left(
\tau \right)  $ can be integrated so that we have the following theorem:

\begin{theorem}
\label{theorem1-5}Suppose $(p(\tau),A(\tau))$ satisfies the Hamiltonian system
(\ref{142-0}). Define
\begin{equation}
F(\tau)\doteqdot A(\tau)+\frac{1}{2}(\zeta(2p(\tau)|\tau)-2\zeta(p(\tau
)|\tau)).\label{F}%
\end{equation}
Then%
\begin{align}
F\left(  \tau \right)   &  =\theta_{1}^{\prime}(\tau)^{\frac{2}{3}}\exp \left \{
\frac{i}{2\pi}\int^{\tau}\left(  2\wp(2p(\hat{\tau})|\hat{\tau})-\wp
(p(\hat{\tau})|\hat{\tau})\right)  d\hat{\tau}\right \}  \times \nonumber \\
&  \left(  \int^{\tau}\frac{-\frac{i}{4\pi}\theta_{1}^{\prime}(\hat{\tau
})^{-\frac{2}{3}}\left(  \sum_{k=0}^{3}n_{k}(n_{k}+1)\wp^{\prime}(p(\hat{\tau
})+\frac{\omega_{k}}{2}|\hat{\tau})\right)  }{\exp \left \{  \frac{i}{2\pi}%
\int^{\hat{\tau}}\left(  2\wp(2p(\tau^{\prime})|\tau^{\prime})-\wp
(p(\tau^{\prime})|\tau^{\prime})\right)  d\tau^{\prime}\right \}  }d\hat{\tau
}+c_{1}\right)  \label{514}%
\end{align}
for some constant $c_{1}\in \mathbb{C}$, where $\theta_{1}^{\prime}(\tau
)=\frac{d\vartheta_{1}\left(  z;\tau \right)  }{dz}|_{z=0}$ and $\vartheta
_{1}\left(  z;\tau \right)  $ is the odd theta function defined in (\ref{ccc}).
In particular, for $n_{k}=0,\forall k$, we have%
\begin{equation}
F(\tau)=c\theta_{1}^{\prime}(\tau)^{\frac{2}{3}}\exp \left \{  \frac{i}{2\pi
}\int^{\tau}\left(  2\wp(2p(\hat{\tau})|\hat{\tau})-\wp(p(\hat{\tau}%
)|\hat{\tau})\right)  d\hat{\tau}\right \}  \label{515}%
\end{equation}
for some constant $c\in \mathbb{C}\backslash \left \{  0\right \}  $.
\end{theorem}

\begin{remark}
\label{remk-1}Let $\eta(\tau)$ be the Dedekind eta function: $\eta
(\tau)\doteqdot q^{\frac{1}{24}}\prod_{n=1}^{\infty}(1-q^{n})$, where
$q=e^{2\pi i\tau}$ for $\tau \in \mathbb{H}$. Then $\theta_{1}^{\prime}%
(\tau)=2\pi \eta^{3}(\tau)$.
\end{remark}

The elliptic form (\ref{124}) and our results above could be applied to
understand the phenomena of collapsing two singular points $\pm p(\tau)$ to
$\frac{\omega_{k}}{2}$ in the generalized Lam\'{e} equation (\ref{89-0}). In
general, when $p(\tau)\rightarrow \frac{\omega_{k}}{2}$ as $\tau \rightarrow
\tau_{0}$, the generalized Lam\'{e} equation might not be well-defined.
However, when $p(\tau)$ is a solution of the elliptic form (\ref{124}), by
using the behavior of $p(\tau)$ near $\tau_{0}$, the following result shows
that the generalized Lam\'{e} equation will converge to the classical Lam\'{e}
equation (\ref{503-1}). 

Observe that if $p(\tau)$ is a solution of the elliptic form (\ref{124}), then
$p(\tau)-\frac{\omega_{k}}{2}$ is also a solution of (\ref{124}) (maybe with
different parameters). Therefore, we only need to study the case
$p(\tau)\rightarrow0$. More precisely, we have:

\begin{theorem}
[=Theorem \ref{thm-II-9}]\label{thm-II-9 copy(1)}Suppose that $n_{k}%
\not \in \frac{1}{2}+\mathbb{Z}$, $k=0,1,2,3$, and (\ref{101-0}) holds. Let
$(p(\tau),A(\tau))$ be a solution of the Hamiltonian system (\ref{142-0}) such
that $p(\tau_{0})=0$ for some $\tau_{0}\in \mathbb{H}$. Then%
\begin{equation}
p(\tau)=c_{0}(\tau-\tau_{0})^{\frac{1}{2}}(1+\tilde{h}(\tau-\tau_{0}%
)+O(\tau-\tau_{0})^{2})\text{ as }\tau \rightarrow \tau_{0},\label{asymp}%
\end{equation}
where $c_{0}^{2}=\pm i\frac{n_{0}+\frac{1}{2}}{\pi}$ and $\tilde{h}%
\in \mathbb{C}$ is some constant. Moreover, the generalized Lam\'{e} equation
(\ref{89-0}) as $\tau \rightarrow \tau_{0}$ converges to%
\[
y^{\prime \prime}=\left[  \sum_{j=1}^{3}n_{j}\left(  n_{j}+1\right)  \wp \left(
z+\frac{\omega_{j}}{2}\right)  +m(m+1)\wp(z)+B_{0}\right]  y\text{ in }%
E_{\tau_{0}}%
\]
where%
\[
m=\left \{
\begin{array}
[c]{l}%
n_{0}+1\text{ \ if \ }c_{0}^{2}=i\frac{n_{0}+\frac{1}{2}}{\pi},\\
n_{0}-1\text{ \ if \ }c_{0}^{2}=-i\frac{n_{0}+\frac{1}{2}}{\pi},
\end{array}
\right.
\]%
\[
B_{0}=2\pi ic_{0}^{2}\left(  4\pi i\tilde{h}-\eta_{1}(\tau_{0})\right)
-\sum_{j=1}^{3}n_{j}(n_{j}+1)e_{j}(\tau_{0}).
\]

\end{theorem}

Theorem \ref{thm-II-9 copy(1)} will be proven in Section 3. In Section 4, we
will give another application of our isomonodromy theory (see Corollary
\ref{corcor}). More precisely, we will establish a one to one correspondence
between the generalized Lam\'{e} equation and the Fuchsian equation on
$\mathbb{CP}^{1}$. Furthermore, we will prove that if one of them is monodromy
preserving then so is the other one. We remark that all the results above have
important applications in our coming paper \cite{Chen-Kuo-Lin}. For example,
Theorem \ref{thm-II-9 copy(1)} can be used to study the converge of even
solutions of the mean field equation (\ref{501}) as $p\left(  \tau \right)
\rightarrow0$ when $\tau \rightarrow \tau_{0}$.\medskip

We conclude this section by comparing our result Theorem \ref{theorem1-2} with
the paper \cite{Kawai} by Kawai. Define the Fuchsian equation in $E_{\tau}$ by%
\begin{equation}
y^{\prime \prime}\left(  z\right)  =q\left(  z\right)  y\left(  z\right)
,\label{OK}%
\end{equation}
where%
\begin{align}
q\left(  z\right)  = &  L+\sum_{i=0}^{m}\left[  H_{i}\zeta \left(  z-t_{i}%
|\tau \right)  +\frac{1}{4}\left(  \theta_{i}^{2}-1\right)  \wp \left(
z-t_{i}|\tau \right)  \right]  \label{OK-1}\\
&  +\sum_{\alpha=0}^{m}\left[  -\mu_{\alpha}\zeta \left(  z-b_{\alpha}%
|\tau \right)  +\frac{3}{4}\wp \left(  z-b_{\alpha}|\tau \right)  \right]
\nonumber
\end{align}
with%
\begin{equation}
\sum_{i=0}^{m}H_{i}-\sum_{\alpha=0}^{m}\mu_{\alpha}=0.\label{OK-2}%
\end{equation}
Here $L$, $H_{i}$, $t_{i}$, $\theta_{i}$, $\mu_{\alpha}$, $b_{\alpha}$ are
complex parameters with $t_{0}=0$. The isomonodromic deformation of equation
(\ref{OK}) was first treated by Okamoto \cite{Okamoto} without varying the
underlying elliptic curves and then generalized by Iwasaki \cite{IW} to the
case of higher genus. Let $\mathcal{R}$ be the space of conjugacy classes of
the monodromy representation of $\pi_{1}((E_{\tau}\backslash S),q_{0})$, where
$S$ denotes the set of singular points. Then it was known that $\mathcal{R}$
is a complex manifold. It was proved by Iwasaki \cite{IW} that there exists a
natural symplectic structure $\Omega$ on the space $\mathcal{R}$. In
\cite{Kawai}, Kawai considered the same Fuchsian equation (\ref{OK}) but
allowed the underlying elliptic curves to vary as well. By using the pull-back
principle to the symplectic 2-form $\Omega$, Kawai studied isomonodromic
deformations for equation (\ref{OK}) which are described as a completely
integrable Hamiltonian system: for $1\leq i\leq m,$%
\begin{equation}
\frac{\partial b_{\alpha}}{\partial t_{i}}=\sum_{i=1}^{m}\frac{\partial H_{i}%
}{\partial \mu_{\alpha}}\text{, }\frac{\partial b_{\alpha}}{\partial \tau}%
=\frac{\partial \mathcal{H}}{\partial \mu_{\alpha}}\text{, }\frac{\partial
\mu_{\alpha}}{\partial t_{i}}=-\sum_{i=1}^{m}\frac{\partial H_{i}}{\partial
b_{\alpha}}\text{, }\frac{\partial \mu_{\alpha}}{\partial \tau}=-\frac
{\partial \mathcal{H}}{\partial b_{\alpha}},\label{Ham}%
\end{equation}
where%
\[
\mathcal{H}=\frac{1}{2\pi i}\left[  L+\eta_{1}(\tau)\left(  \sum_{\alpha
=0}^{m}b_{\alpha}\mu_{\alpha}-\sum_{i=1}^{m}t_{i}H_{i}\right)  \right]  .
\]
Now considering the simplest case $m=0$ and by using (\ref{OK-2}) and
$t_{0}=0$, the potential $q(z)$ takes the simple form (the subscript $0$ is
dropped for simplicity)%
\begin{equation}
q\left(  z\right)  =L+\mu \zeta \left(  z|\tau \right)  +\frac{1}{4}\left(
\theta^{2}-1\right)  \wp \left(  z|\tau \right)  -\mu \zeta \left(  z-b|\tau
\right)  +\frac{3}{4}\wp \left(  z-b|\tau \right)  .\label{potential}%
\end{equation}
Consequently, the Hamiltonian system (\ref{Ham}) is reduced to%
\begin{equation}
\left \{
\begin{array}
[c]{l}%
\frac{db}{d\tau}=\frac{-i}{2\pi}\left[  2\mu-\zeta(b|\tau)+b\eta_{1}\right]
,\\
\\
\frac{d\mu}{d\tau}=\frac{i}{2\pi}\left[  \mu \wp(b|\tau)+\mu \eta_{1}-\frac
{1}{4}(\theta^{2}-1)\wp^{\prime}(b|\tau)\right]  .
\end{array}
\right.  \label{Ham-1}%
\end{equation}
Furthermore, the Hamiltonian system (\ref{Ham-1}) is equivalent to%
\[
\frac{d^{2}}{d\tau^{2}}\left(  \frac{b}{2}\right)  =-\frac{1}{4\pi^{2}}%
\sum_{k=0}^{3}\frac{\theta^{2}}{32}\wp^{\prime}\left(  \frac{b}{2}%
+\frac{\omega_{k}}{2}|\tau \right)  ,
\]
which implies that $\frac{b}{2}$ satisfies the elliptic form (\ref{124}) with
$\left(  \alpha_{0},\alpha_{1},\alpha_{2},\alpha_{3}\right)  =(\frac
{\theta^{2}}{32},\frac{\theta^{2}}{32},\frac{\theta^{2}}{32},\frac{\theta^{2}%
}{32})$; see \cite[Theorem 3]{Kawai}. It is clear that our potential $I\left(
z\right)  $ is \textit{different from} (\ref{potential}) except for
$n_{k}=0,k=0,1,2,3$ in $I\left(  z\right)  $ and $\theta=\pm2$ in
(\ref{potential}). Notice that the linear ODE (\ref{OK}) with (\ref{potential}%
) only has the apparent singularity at $b$. Thus it seems that the monodromy
representation for (\ref{OK}) could not be reduced to $\pi_{1}\left(  E_{\tau
}\right)  $ when $\theta \not =\pm2$. However, when $n_{k}\in \mathbb{N}%
\cup \left \{  0\right \}  $ for all $k$, the monodromy representation for
(\ref{89-0}) could be simplified. We remark it is an advantage when we study
the elliptic form (\ref{124}) with $\alpha_{k}=\frac{1}{2}\left(  n_{k}%
+\frac{1}{2}\right)  ^{2},\ k=0,1,2,3$. From Kawai's result in \cite{Kawai}
and ours, it can be seen that the elliptic form (\ref{124}) governs
isomonodromic deformations of different linear ODEs (e.g. (\ref{OK}) with
(\ref{potential}) and (\ref{89-0})). Therefore, it is important to choose a
suitable linear ODE when generic parameters\textit{ }are considered.

\section{Painlev\'{e} VI and Hamiltonian system on the moduli space}

In this section, we want to develop an isomonodromy theory on the moduli space
of elliptic curves. For this purpose, there are two fundamental issues needed
to be discussed: (i) to derive the Hamiltonian system for the isomonodromic
deformation of the generalized Lam\'{e} equation (\ref{89-0}) with
$n_{i}\not \in \frac{1}{2}+\mathbb{Z}$, $i=0,1,2,3$; (ii) to prove the
equivalence between the Hamiltonian system and the elliptic form (\ref{124}).
We remark that the (ii) part holds true without any condition. Recall the
generalized Lam\'{e} equation defined by
\begin{equation}
y^{\prime \prime}=I\left(  z;\tau \right)  y, \label{152}%
\end{equation}
where%
\begin{align}
I\left(  z;\tau \right)  =  &  \sum_{i=0}^{3}n_{i}\left(  n_{i}+1\right)
\wp \left(  z+\frac{\omega_{i}}{2}|\tau \right)  +\frac{3}{4}\left(  \wp \left(
z+p|\tau \right)  +\wp \left(  z-p|\tau \right)  \right) \nonumber \\
&  +A\left(  \zeta \left(  z+p|\tau \right)  -\zeta \left(  z-p|\tau \right)
\right)  +B, \label{154}%
\end{align}
and $p\left(  \tau \right)  \not \in E_{\tau}\left[  2\right]  $. By replacing
$n_{i}$ by $-n_{i}-1$ if necessary, we always assume $n_{i}\geq-\frac{1}{2}$
for all $i$. Remark that, since we assume $n_{i}\not \in \frac{1}%
{2}+\mathbb{Z}$, the exponent difference of (\ref{152}) at $\frac{\omega_{i}%
}{2}$ is $2n_{i}+1\not \in 2\mathbb{Z}$, implying that (\ref{152}) has no
logarithmic singularity at $\frac{\omega_{i}}{2}$, $0\leq i\leq3$.

For equation (\ref{152}), the necessary and sufficient condition for apparent
singularity at $\pm p$ is given by

\begin{lemma}
\label{lem-apparent}$\pm p$ are apparent singularities of (\ref{152}) iff $A$
and $B$ satisfy%
\begin{equation}
B=A^{2}-\zeta \left(  2p\right)  A-\frac{3}{4}\wp \left(  2p\right)  -\sum
_{i=0}^{3}n_{i}\left(  n_{i}+1\right)  \wp \left(  p+\frac{\omega_{i}}%
{2}\right)  . \label{101}%
\end{equation}

\end{lemma}

\begin{proof}
It suffices to prove this lemma for the point $p$. Let $y_{i}$, $i=1,2$, be
two linearly independent solutions to (\ref{152}). Define $f\doteqdot
\frac{y_{1}}{y_{2}}$ as a ratio of two independent solutions and
$v\doteqdot \log f^{\prime}$. Then%
\begin{equation}
\{f;z\}=v^{\prime \prime}-\frac{1}{2}\left(  v^{\prime}\right)  ^{2}=-2I(z).
\label{78}%
\end{equation}
It is obvious that (\ref{152}) has no solutions with logarithmic singularity
at $p$ iff $f\left(  z\right)  $ has no logarithmic singularity at $p$. First
we prove the necessary part. Without loss of generality, we may assume
$f\left(  z\right)  $ is holomorphic at $p$. The local expansion of $f$ at $p$
is:%
\[
f\left(  z\right)  =c_{0}+c_{2}\left(  z-p\right)  ^{2}+\cdot \cdot \cdot,
\]%
\begin{equation}
v\left(  z\right)  =\log f^{\prime}\left(  z\right)  =\log2c_{2}+\log \left(
z-p\right)  +\sum_{j\geq1}d_{j}\left(  z-p\right)  ^{j}, \label{78-0}%
\end{equation}%
\begin{equation}
v^{\prime}\left(  z\right)  =\frac{1}{z-p}+\sum_{j\geq0}\tilde{e}_{j}\left(
z-p\right)  ^{j}, \label{78-1}%
\end{equation}%
\[
v^{\prime \prime}\left(  z\right)  =\frac{-1}{\left(  z-p\right)  ^{2}}%
+\sum_{j\geq0}\left(  j+1\right)  \tilde{e}_{j+1}\left(  z-p\right)  ^{j},
\]
where $\tilde{e}_{j}=\left(  j+1\right)  d_{j+1}$. Thus,{\allowdisplaybreaks%
\begin{align}
v^{\prime \prime}-\frac{1}{2}\left(  v^{\prime}\right)  ^{2}=  &  \frac
{-1}{\left(  z-p\right)  ^{2}}+\sum_{j\geq0}\left(  j+1\right)  \tilde
{e}_{j+1}\left(  z-p\right)  ^{j}\label{78-2}\\
&  -\frac{1}{2}\Big[\frac{1}{z-p}+\sum_{j\geq0}\tilde{e}_{j}\left(
z-p\right)  ^{j}\Big]^{2}.\nonumber
\end{align}
}Recalling $I(z)$ in (\ref{154}), we compare both sides of (\ref{78}). The
$\left(  z-p\right)  ^{-2}$ terms match automatically. For the $\left(
z-p\right)  ^{-1}$ term, we get%
\begin{equation}
-\tilde{e}_{0}=2A. \label{79}%
\end{equation}
For the $\left(  z-p\right)  ^{0}$, i.e. the constant term, we
have{\allowdisplaybreaks%
\begin{align}
&  \tilde{e}_{1}-\frac{2}{2}\tilde{e}_{1}-\frac{1}{2}\tilde{e}_{0}%
^{2}\label{80}\\
=  &  -2\sum_{i=0}^{3}n_{i}\left(  n_{i}+1\right)  \wp \left(  p+\frac
{\omega_{i}}{2}\right)  -\frac{3}{2}\wp \left(  2p\right)  -2A\zeta \left(
2p\right)  -2B.\nonumber
\end{align}
}Then (\ref{101}) follows from (\ref{79}) and (\ref{80}) immediately.

For the sufficient part, if (\ref{101}) holds, then $\tilde{e}_{0}$ is given
by (\ref{79}). By any choice of $\tilde{e}_{1}$ and comparing (\ref{78}) and
(\ref{78-2}), $\tilde{e}_{j}$ is determined for all $j\geq2$. Then it follows
from (\ref{78-0})-(\ref{78-1}) that $f\left(  z\right)  $ is holomorphic at
$p$. Since its Schwarzian derivative satisfies (\ref{78}), $f$ is a ratio of
two linearly independent solutions of (\ref{152}). This implies that
(\ref{152}) has no solutions with logarithmic singularity at $p$, namely $p$
is an apparent singularity.
\end{proof}

\subsection{Isomonodromic equation and Hamiltonian system}

The 2nd order generalized Lam\'{e} equation (\ref{152}) can be written into a
1st order linear system%
\begin{equation}
\frac{d}{dz}Y=Q\left(  z;\tau \right)  Y\text{ \ in }E_{\tau},\label{153}%
\end{equation}
where%
\begin{equation}
Q\left(  z;\tau \right)  =\left(
\begin{matrix}
0 & 1\\
I\left(  z;\tau \right)   & 0
\end{matrix}
\right)  .\label{176}%
\end{equation}
The isomonodromic deformation of the generalized Lam\'{e} equation (\ref{152})
is equivalent to the isomonodromic deformation of the linear system
(\ref{153}). Let $y_{1}\left(  z;\tau \right)  $ and $y_{2}\left(
z;\tau \right)  $ be two linearly independent solutions of (\ref{152}), then
$Y\left(  z;\tau \right)  =\left(
\begin{matrix}
y_{1}\left(  z;\tau \right)   & y_{2}\left(  z;\tau \right)  \\
y_{1}^{\prime}\left(  z;\tau \right)   & y_{2}^{\prime}\left(  z;\tau \right)
\end{matrix}
\right)  $ is a fundamental system of solutions to (\ref{153}). In general,
$Y\left(  z;\tau \right)  $ is multi-valued with respect to $z$ and for each
$\tau \in \mathbb{H}$, $Y\left(  z;\tau \right)  $ might have branch points at
$S\doteqdot \left \{  \pm p,\frac{\omega_{k}}{2}|\text{ }k=0,1,2,3\right \}  $.
The fundamental solution $Y\left(  z;\tau \right)  $ is called $M$-invariant
($M$ stands for monodromy) if there is some $q_{0}\in E_{\tau}\backslash S$
and for any loop $\ell \in \pi_{1}(E_{\tau}\backslash S$, $q_{0})$, there exists
$\rho \left(  \ell \right)  \in SL\left(  2,\mathbb{C}\right)  $
\textit{independent of }$\tau$ such that%
\[
\ell^{\ast}Y\left(  z;\tau \right)  =Y\left(  z;\tau \right)  \rho \left(
\ell \right)
\]
holds for $z$ near the base point $q_{0}$. Here $\ell^{\ast}Y\left(
z;\tau \right)  $ denotes the analytic continuation of $Y\left(  z;\tau \right)
$ along $\ell$. Let $\gamma_{k}\in \pi_{1}(E_{\tau}\backslash S,q_{0}),$
$k=0,1,2,3,\pm$, be simple loops which encircle the singularties $\frac
{\omega_{k}}{2},$ $k=0,1,2,3$ and $\pm p$ once respectively, and $\ell_{j}%
\in \pi_{1}(E_{\tau}\backslash S,q_{0})$, $j=1,2$, be two fundamental cycles of
$E_{\tau}$ such that its lifting in $\mathbb{C}$ is a straight line connecting
$q_{0}$ and $q_{0}+\omega_{j}$. We also require these lines do not pass any
singularties. Of course, all the pathes do not intersect with each other
except at $q_{0}$. We note that when $\tau$ varies in a neighborhood of some
$\tau_{0}$, $\gamma_{k}$ and $\ell_{1}$ can be choosen independent of $\tau$.

Clearly the monodromy group with respect to $Y\left(  z;\tau \right)  $ is
generated by $\left \{  \rho \left(  \ell_{j}\right)  ,\rho \left(  \gamma
_{k}\right)  |\text{ }j=1,2\text{ and }k=0,1,2,3,\pm \right \}  $. Thus,
$Y\left(  z;\tau \right)  $ is $M$-invariant if and only if the matrices
$\rho \left(  \ell_{j}\right)  ,\rho \left(  \gamma_{k}\right)  $ are
independent of $\tau$. Notice that $I\left(  \cdot;\tau \right)  $ is an
elliptic function, so we can also treat (\ref{153}) as a equation defined in
$\mathbb{C}$, i.e.,%
\begin{equation}
\frac{d}{dz}Y=Q\left(  z;\tau \right)  Y\text{ \ in }\mathbb{C}. \label{153-1}%
\end{equation}
Furthermore, we can identify solutions of (\ref{153}) and (\ref{153-1}) in an
obvious way. For example, after analytic continuation, any solution
$Y(\cdot;\tau)$ of (\ref{153}) can be extended to be a solution of
(\ref{153-1}) as a multi-valued matrix function defined in $\mathbb{C}$ (still
denote it by $Y(\cdot;\tau)$). In the sequel, we always identify solutions of
(\ref{153}) and (\ref{153-1}). Then we have the following theorem:

\begin{theorem}
\label{thm M}System (\ref{153}) is monodromy preserving as $\tau$ deforms if
and only if there exists a single-valued matrix function $\Omega \left(
z;\tau \right)  $ defined in $\mathbb{C}\times \mathbb{H}$ satisfying%
\begin{equation}
\left \{
\begin{array}
[c]{l}%
\Omega \left(  z+1;\tau \right)  =\Omega \left(  z;\tau \right) \\
\Omega \left(  z+\tau;\tau \right)  =\Omega \left(  z;\tau \right)  -Q\left(
z;\tau \right)  ,
\end{array}
\right.  \label{204}%
\end{equation}
such that the following Pfaffian\emph{ }system%
\begin{equation}
\left \{
\begin{array}
[c]{l}%
\frac{\partial}{\partial z}Y(z;\tau)=Q\left(  z;\tau \right)  Y(z;\tau
)\smallskip \\
\frac{\partial}{\partial \tau}Y(z;\tau)=\Omega \left(  z;\tau \right)  Y(z;\tau)
\end{array}
\right.  \text{ \ in \ }\mathbb{C}\times \mathbb{H} \label{184}%
\end{equation}
is completely integrable.
\end{theorem}

\begin{remark}
\label{remk}The classical isomonodromy theory in $\mathbb{C}$ (see e.g.
\cite[Proposition 3.1.5]{GP}) says that system (\ref{153-1}) is monodromy
preserving if and only if there exists a single-valued matrix function
$\Omega \left(  z;\tau \right)  $ defined in $\mathbb{C}\times \mathbb{H}$ such
that (\ref{184}) is completely integrable. Theorem \ref{thm M} is the
counterpart of this classical theory in the torus $E_{\tau}$. The property
(\ref{204}) comes from the preserving of monodromy matrices $\rho \left(
\ell_{j}\right)  $, $j=1,2$ during the deformation (see from the proof of
Theorem \ref{thm M} below). Notice that $\rho \left(  \ell_{j}\right)  $ can be
considered as connection matrices along the straight line $\ell_{j}$
connecting $q_{0}$ and $q_{0}+\omega_{j}$ for system (\ref{153-1}).
\end{remark}

Notice that system (\ref{184}) is completely integrable if and only if
\begin{equation}
\frac{\partial}{\partial \tau}Q\left(  z;\tau \right)  =\frac{\partial}{\partial
z}\Omega \left(  z;\tau \right)  +\left[  \Omega \left(  z;\tau \right)  ,Q\left(
z;\tau \right)  \right]  ,\text{ and } \label{185}%
\end{equation}%
\begin{equation}
d(\Omega \left(  z;\tau \right)  d\tau)=\left[  \Omega \left(  z;\tau \right)
d\tau \right]  \wedge \left[  \Omega \left(  z;\tau \right)  d\tau \right]  ,
\label{185-1}%
\end{equation}
where $d$ denotes the exterior differentiation with respect to $\tau$ in
(\ref{185-1}). See Lemma 3.14 in \cite{GP} for the proof. Clearly
(\ref{185-1}) holds automatically since there is only one deformation
parameter. We need the following lemma to prove Theorem \ref{thm M}.

\begin{lemma}
\label{lemma}Let $Y\left(  z;\tau \right)  $ be an $M$-invariant fundamental
solution of system (\ref{153}) and define a $2\times2$ matrix-valued function
$\Omega \left(  z;\tau \right)  $ in $E_{\tau}$ by%
\begin{equation}
\Omega \left(  z;\tau \right)  =\frac{\partial}{\partial \tau}Y\cdot Y^{-1}.
\label{w}%
\end{equation}
Then $\Omega \left(  z;\tau \right)  $ can be extended to be a globally defined
matrix-valued function in $\mathbb{C\times H}$ by analytic continuation (still
denote it by $\Omega \left(  z;\tau \right)  $). In particular, (\ref{w}) holds
in $\mathbb{C\times H}$ by considering $Y\left(  z;\tau \right)  $ as a
solution of system (\ref{153-1}).
\end{lemma}

\begin{proof}
The proof is the same as that in the classical isomonodromy theory in
$\mathbb{C}$. Indeed, since $Y\left(  z;\tau \right)  $ is $M$-invariant, we
have{\allowdisplaybreaks
\begin{align}
\gamma_{k}^{\ast}\Omega \left(  z;\tau \right)   &  =\gamma_{k}^{\ast}\left(
\frac{\partial}{\partial \tau}Y\cdot Y^{-1}\right)  =\frac{\partial}%
{\partial \tau}\gamma_{k}^{\ast}Y\cdot \gamma_{k}^{\ast}Y^{-1}\label{qqq6}\\
&  =\frac{\partial}{\partial \tau}\left(  Y\rho \left(  \gamma_{k}\right)
\right)  \cdot \rho \left(  \gamma_{k}\right)  ^{-1}Y^{-1}\nonumber \\
&  =\frac{\partial}{\partial \tau}Y\cdot Y^{-1}=\Omega \left(  z;\tau \right)
\nonumber
\end{align}
}for $k=0,1,2,3,\pm$, namely $\Omega \left(  \cdot;\tau \right)  $ is invariant
under the analytic continuation along $\gamma_{k}$. Thus, $\Omega \left(
\cdot;\tau \right)  $ is single-valued in any fundamental domain of $E_{\tau}$
for each $\tau$. Then for each $\tau \in \mathbb{H}$, we could extend
$\Omega \left(  z;\tau \right)  $ to be a globally defined matrix-valued
function in $\mathbb{C}$ by analytic continuation.
\end{proof}

From now on, we consider equation (\ref{153}) defined in $\mathbb{C}$, i.e.,
(\ref{153-1}). The analytic continuation along any curve in $\mathbb{C}$
always keep the relation (\ref{w}) between $Y\left(  z;\tau \right)  $ and
$\Omega \left(  z;\tau \right)  $.

\begin{proof}
[Proof of Theorem \ref{thm M}]First we prove the necessary part. Let $Y\left(
z;\tau \right)  $ be an $M$-invariant fundamental solution of system
(\ref{153}) and define $\Omega \left(  z;\tau \right)  $ by $Y\left(
z;\tau \right)  $. By Lemma \ref{lemma}, $\Omega \left(  z;\tau \right)  $ is a
single-valued matrix function in $\mathbb{C}\times \mathbb{H}$ and $Y(z;\tau)$
is a solution of (\ref{184}), which implies (\ref{185}). Hence the
Pfaffian\emph{ }system (\ref{184}) is completely integrable.

It suffices to prove that $\Omega(z;\tau)$ satisfies (\ref{204}). Note that
$\Omega(z;\tau)$ is single-valued in $\mathbb{C\times H}$. Therefore, to prove
(\ref{204}), we only need to prove its validity in a small neighborhood
$U_{q_{0}}\times V_{\tau_{0}}$ of some $(q_{0},\tau_{0})$, where $q_{0}$ is
the base point. By considering $Y\left(  z;\tau \right)  $ as a solution of
system (\ref{153-1}), we see from Remark \ref{remk} and Lemma \ref{lemma}
that, for $(z,\tau)\in U_{q_{0}}\times V_{\tau_{0}}$,%
\begin{equation}
Y(z+\omega_{i};\tau)=Y(z;\tau)\rho \left(  \ell_{i}\right)  , \label{1001}%
\end{equation}%
\begin{equation}
\Omega \left(  z;\tau \right)  =\frac{\partial}{\partial \tau}Y(z;\tau)\cdot
Y(z;\tau)^{-1}, \label{1002}%
\end{equation}%
\begin{equation}
\Omega \left(  z+\omega_{i};\tau \right)  =\frac{\partial}{\partial \tau
}Y(z+\omega_{i};\tau)\cdot Y(z+\omega_{i};\tau)^{-1}. \label{1003}%
\end{equation}
Therefore, (\ref{1001}) and (\ref{1003}) give{\allowdisplaybreaks
\begin{align*}
&  \Omega \left(  z+\omega_{i};\tau \right) \\
&  =\left[  \frac{d}{d\tau}Y\left(  z+\omega_{i};\tau \right)  -\frac{\partial
}{\partial z}Y\left(  z+\omega_{i};\tau \right)  \frac{d}{d\tau}\omega
_{i}\right]  \cdot Y\left(  z+\omega_{i};\tau \right)  ^{-1}\\
&  =\left[  \frac{d}{d\tau}\left(  Y\left(  z;\tau \right)  \rho \left(
\ell_{i}\right)  \right)  -\frac{\partial}{\partial z}\left(  Y\left(
z;\tau \right)  \rho \left(  \ell_{i}\right)  \right)  \frac{d}{d\tau}\omega
_{i}\right]  \cdot \left(  Y\left(  z;\tau \right)  \rho \left(  \ell_{i}\right)
\right)  ^{-1}.
\end{align*}
}Since $\rho \left(  \ell_{i}\right)  ,i=1,2$, are independent of $\tau,$ we
have%
\[
\Omega \left(  z+1;\tau \right)  =\frac{d}{d\tau}Y\left(  z;\tau \right)  \cdot
Y\left(  z;\tau \right)  ^{-1}=\text{$\Omega \left(  z;\tau \right)  ,$}%
\]
and%
\begin{align*}
\Omega \left(  z+\tau;\tau \right)   &  =\frac{d}{d\tau}Y\left(  z;\tau \right)
\cdot Y\left(  z;\tau \right)  ^{-1}-\frac{\partial}{\partial z}Y\left(
z;\tau \right)  \cdot Y\left(  z;\tau \right)  ^{-1}\\
&  =\text{$\Omega \left(  z;\tau \right)  $}-Q\left(  z;\tau \right)  .
\end{align*}
This proves (\ref{204}).

Conversely, suppose there exists a single-valued matrix function
$\Omega \left(  z;\tau \right)  $ in $\mathbb{C}\times \mathbb{H}$ satisfying
(\ref{204}) such that (\ref{184}) is completely integrable. Let $Y\left(
z;\tau \right)  $ be a solution of the Pfaffian\emph{ }system (\ref{184}). Then
(\ref{1001})-(\ref{1003}) hold and $Y\left(  z;\tau \right)  $ satisfies system
(\ref{153}) in $E_{\tau}$. Hence%
\[
\frac{\partial}{\partial \tau}Y\left(  z+\omega_{i};\tau \right)  =\frac
{d}{d\tau}Y\left(  z+\omega_{i};\tau \right)  -\frac{\partial}{\partial
z}Y\left(  z+\omega_{i};\tau \right)  \frac{d}{d\tau}\omega_{i},
\]
which implies{\allowdisplaybreaks
\begin{align}
\frac{\partial}{\partial \tau}Y\left(  z+1;\tau \right)   &  =\frac{d}{d\tau
}\left(  Y\left(  z;\tau \right)  \rho \left(  \ell_{1}\right)  \right)
\label{qqq2}\\
&  =\frac{\partial}{\partial \tau}Y\left(  z;\tau \right)  \cdot \rho \left(
\ell_{1}\right)  +Y\left(  z;\tau \right)  \frac{d}{d\tau}\rho \left(  \ell
_{1}\right) \nonumber \\
&  =\text{$\Omega \left(  z;\tau \right)  $}Y\left(  z;\tau \right)  \rho \left(
\ell_{1}\right)  +Y\left(  z;\tau \right)  \frac{d}{d\tau}\rho \left(  \ell
_{1}\right) \nonumber
\end{align}
}and {\allowdisplaybreaks
\begin{align}
&  \frac{\partial}{\partial \tau}Y\left(  z+\tau;\tau \right) \label{qqq3}\\
&  =\frac{d}{d\tau}\left(  Y\left(  z;\tau \right)  \rho \left(  \ell
_{2}\right)  \right)  -\frac{\partial}{\partial z}\left(  Y\left(
z;\tau \right)  \rho \left(  \ell_{2}\right)  \right) \nonumber \\
&  =\frac{\partial}{\partial \tau}Y\left(  z;\tau \right)  \cdot \rho \left(
\ell_{2}\right)  +Y\left(  z;\tau \right)  \frac{d}{d\tau}\rho \left(  \ell
_{2}\right)  -\frac{\partial}{\partial z}Y\left(  z;\tau \right)  \cdot
\rho \left(  \ell_{2}\right) \nonumber \\
&  =\left[  \text{$\Omega \left(  z;\tau \right)  -Q\left(  z;\tau \right)  $%
}\right]  Y\left(  z;\tau \right)  \rho \left(  \ell_{2}\right)  +Y\left(
z;\tau \right)  \frac{d}{d\tau}\rho \left(  \ell_{2}\right)  .\nonumber
\end{align}
}On the other hand, by (\ref{1001}) and (\ref{1003}), we also have%
\begin{equation}
\frac{\partial}{\partial \tau}Y\left(  z+\omega_{i};\tau \right)  =\text{$\Omega
\left(  z+\omega_{i};\tau \right)  $}Y\left(  z;\tau \right)  \rho \left(
\ell_{i}\right)  . \label{qqq4}%
\end{equation}
Then by (\ref{qqq2}), (\ref{qqq3}), (\ref{qqq4}) and (\ref{204}), we have%
\[
Y\left(  z;\tau \right)  \frac{d}{d\tau}\rho \left(  \ell_{1}\right)  =Y\left(
z;\tau \right)  \frac{d}{d\tau}\rho \left(  \ell_{2}\right)  =0.
\]
Also, by the same argument as (\ref{qqq6}), we could prove%
\[
Y\left(  z;\tau \right)  \frac{d}{d\tau}\rho \left(  \gamma_{k}\right)
=0\text{, }k=0,1,2,3,\pm.
\]
Because of $\det Y\not =0$, we conclude that%
\[
\frac{d}{d\tau}\rho \left(  \ell_{j}\right)  =\frac{d}{d\tau}\rho \left(
\gamma_{k}\right)  =0.
\]
Thus, $Y$ is an $M$-invariant solution of (\ref{153}). That is, system
(\ref{153}) is monodromy preserving. This completes the proof.
\end{proof}

Write $\Omega \left(  z;\tau \right)  =\left(
\begin{matrix}
\Omega_{11} & \Omega_{12}\\
\Omega_{21} & \Omega_{22}%
\end{matrix}
\right)  $. Since $Q\left(  z;\tau \right)  $ has the special form (\ref{176}),
by a straightforward computation, the integrability condition (\ref{185}) is
equivalent to%
\begin{equation}
\Omega_{12}^{\prime \prime \prime}-4I\Omega_{12}^{\prime}-2I^{\prime}\Omega
_{12}+2\frac{\partial}{\partial \tau}I=0\text{ \ in }\mathbb{C\times
H},\label{186}%
\end{equation}
where we denote $^{\prime}=\frac{\partial}{\partial z}$ to be the partial
derivative with respect to the variable $z$. This computation is the same as
the case in $\mathbb{C}$ (see e.g. \cite[Proposition 3.5.1]{GP}), so we omit
the details. Then we have the following fundamental theorem for isomonodromic
deformations of (\ref{153}) in the moduli space of elliptic curves:

\begin{theorem}
\label{thm M1}System (\ref{153}) is monodromy preserving as $\tau$ deforms if
and only if there exists a single-valued solution $\Omega_{12}\left(
z;\tau \right)  $ to (\ref{186}) satisfying
\begin{align}
\Omega_{12}\left(  z+1;\tau \right)   &  =\Omega_{12}\left(  z;\tau \right)
,\label{qqq}\\
\Omega_{12}\left(  z+\tau;\tau \right)   &  =\Omega_{12}\left(  z;\tau \right)
-1.\nonumber
\end{align}

\end{theorem}

\begin{proof}
By Theorem \ref{thm M}, it suffices to prove the sufficient part. Suppose
there exists a single-valued solution $\Omega_{12}\left(  z;\tau \right)  $ to
(\ref{186}) satisfying (\ref{qqq}). Then we define $\Omega \left(
z;\tau \right)  =\left(
\begin{matrix}
\Omega_{11} & \Omega_{12}\\
\Omega_{21} & \Omega_{22}%
\end{matrix}
\right)  $ by setting{\allowdisplaybreaks
\begin{align}
\Omega_{11}\left(  z;\tau \right)   &  =-\frac{1}{2}\Omega_{12}^{\prime}\left(
z;\tau \right)  ,\label{qq}\\
\Omega_{21}\left(  z;\tau \right)   &  =\Omega_{11}^{\prime}\left(
z;\tau \right)  +\Omega_{12}\left(  z;\tau \right)  I\left(  z;\tau \right)
,\nonumber \\
\Omega_{22}\left(  z;\tau \right)   &  =\Omega_{12}^{\prime}\left(
z;\tau \right)  +\Omega_{11}\left(  z;\tau \right)  .\nonumber
\end{align}
}By (\ref{186}), it is easy to see that $\Omega \left(  z;\tau \right)  $
satisfies the integrability condition (\ref{185}) (see e.g. \cite[Proposition
3.5.1]{GP}), namely (\ref{184}) is completely integrable. Finally, (\ref{204})
follows from (\ref{qqq}). This completes the proof.
\end{proof}

The first main result of this section is as follows:

\begin{theorem}
\label{thm4-2}Let $n_{k}\not \in \frac{1}{2}+\mathbb{Z}$, $k=0,1,2,3$ and
$p\left(  \tau \right)  $ is an apparent singular point of the generalized
Lam\'{e} equation (\ref{152}) with (\ref{154}). Then (\ref{152}) with $\left(
p,A\right)  =\left(  p\left(  \tau \right)  ,A\left(  \tau \right)  \right)  $
is an isomonodromic deformation with respect to $\tau$ if and only if $\left(
p\left(  \tau \right)  ,A\left(  \tau \right)  \right)  $ satisfies the
Hamiltonian system:%
\begin{equation}
\frac{dp\left(  \tau \right)  }{d\tau}=\frac{\partial K\left(  p,A,\tau \right)
}{\partial A},\text{ \ }\frac{dA\left(  \tau \right)  }{d\tau}=-\frac{\partial
K\left(  p,A,\tau \right)  }{\partial p}, \label{142}%
\end{equation}
where%
\begin{equation}
K\left(  p,A,\tau \right)  =\frac{-i}{4\pi}\left(
\begin{array}
[c]{l}%
A^{2}+\left(  -\zeta \left(  2p|\tau \right)  +2p\eta_{1}\left(  \tau \right)
\right)  A-\frac{3}{4}\wp \left(  2p|\tau \right) \\
-\sum_{k=0}^{3}n_{k}\left(  n_{k}+1\right)  \wp \left(  p+\frac{\omega_{k}}%
{2}|\tau \right)
\end{array}
\right)  . \label{143}%
\end{equation}

\end{theorem}

To prove Theorem \ref{thm4-2}, we need the following formulae for theta
functions and functions in Weierstrass elliptic function theory.

\begin{lemma}
\label{lem4-1} The following formulae hold:

(i)%
\[
\frac{\partial}{\partial \tau}\ln \sigma \left(  z|\tau \right)  =\frac{i}{4\pi
}\left[  \wp \left(  z|\tau \right)  -\zeta^{2}\left(  z|\tau \right)  +2\eta
_{1}\left(  z\zeta \left(  z|\tau \right)  -1\right)  -\frac{1}{12}g_{2}%
z^{2}\right]  ,
\]
(ii)%
\[
\frac{\partial}{\partial \tau}\zeta \left(  z|\tau \right)  =\frac{i}{4\pi
}\left[
\begin{array}
[c]{l}%
\wp^{\prime}\left(  z|\tau \right)  +2\left(  \zeta \left(  z|\tau \right)
-z\eta_{1}\left(  \tau \right)  \right)  \wp \left(  z|\tau \right)  \\
+2\eta_{1}\zeta \left(  z|\tau \right)  -\frac{1}{6}zg_{2}\left(  \tau \right)
\end{array}
\right]  ,
\]
(iii)%
\[
\frac{\partial}{\partial \tau}\wp \left(  z|\tau \right)  =\frac{-i}{4\pi}\left[
\begin{array}
[c]{l}%
2\left(  \zeta \left(  z|\tau \right)  -z\eta_{1}\left(  \tau \right)  \right)
\wp^{\prime}\left(  z|\tau \right)  \\
+4\left(  \wp \left(  z|\tau \right)  -\eta_{1}\right)  \wp \left(
z|\tau \right)  -\frac{2}{3}g_{2}\left(  \tau \right)
\end{array}
\right]  ,
\]
(iv)%
\[
\frac{\partial}{\partial \tau}\wp^{\prime}\left(  z|\tau \right)  =\frac
{-i}{4\pi}\left[
\begin{array}
[c]{l}%
6\left(  \wp \left(  z|\tau \right)  -\eta_{1}\right)  \wp^{\prime}\left(
z|\tau \right)  \\
+\left(  \zeta \left(  z|\tau \right)  -z\eta_{1}\left(  \tau \right)  \right)
\left(  12\wp^{2}\left(  z|\tau \right)  -g_{2}\left(  \tau \right)  \right)
\end{array}
\right]  ,
\]
(v)%
\[
\frac{d}{d\tau}\eta_{1}\left(  \tau \right)  =\frac{i}{4\pi}\left[  2\eta
_{1}^{2}-\frac{1}{6}g_{2}\left(  \tau \right)  \right]  ,
\]
(vi)%
\[
\frac{d}{d\tau}\ln \theta_{1}^{\prime}\left(  \tau \right)  =\frac{3i}{4\pi}%
\eta_{1},
\]
where%
\[
g_{2}\left(  \tau \right)  =-4\left(  e_{1}\left(  \tau \right)  e_{2}\left(
\tau \right)  +e_{1}\left(  \tau \right)  e_{3}\left(  \tau \right)
+e_{2}\left(  \tau \right)  e_{3}\left(  \tau \right)  \right)  ,
\]%
\[
\theta_{1}^{\prime}\left(  \tau \right)  \doteqdot \frac{d}{dz}\vartheta
_{1}\left(  z;\tau \right)  |_{z=0},\text{ \ }\frac{d}{dz}\ln \sigma \left(
z|\tau \right)  \doteqdot \zeta \left(  z|\tau \right)  ,
\]%
\begin{equation}
\vartheta_{1}\left(  z;\tau \right)  \doteqdot-i\sum_{n=-\infty}^{\infty
}(-1)^{n}e^{(n+\frac{1}{2})^{2}\pi i\tau}e^{(2n+1)\pi iz}.\label{ccc}%
\end{equation}

\end{lemma}

Those formulae in Lemma \ref{lem4-1} are known in the literature; see e.g.
\cite{YB} and references therein for the proofs.

To give a motivation for our proof of Theorem \ref{thm4-2}, we first consider
the simplest case $n_{k}=0,\forall k$: Let $a_{1}=r+s\tau$ where $\left(
r,s\right)  \in \mathbb{C}^{2}\backslash \frac{1}{2}\mathbb{Z}^{2}$ is a fixed
pair and $\pm p\left(  \tau \right)  $, $A\left(  \tau \right)  $, $B\left(
\tau \right)  $ be defined by%
\begin{equation}
\zeta \left(  a_{1}\left(  \tau \right)  +p\left(  \tau \right)  \right)
+\zeta \left(  a_{1}\left(  \tau \right)  -p\left(  \tau \right)  \right)
-2\left(  r\eta_{1}(\tau)+s\eta_{2}(\tau)\right)  =0, \label{132}%
\end{equation}%
\begin{equation}
A=\frac{1}{2}\left[  \zeta \left(  p+a_{1}\right)  +\zeta \left(  p-a_{1}%
\right)  -\zeta \left(  2p\right)  \right]  , \label{44}%
\end{equation}%
\begin{equation}
B=A^{2}-\zeta \left(  2p\right)  A-\frac{3}{4}\wp \left(  2p\right)  ,
\label{45}%
\end{equation}
respectively. In \cite{Chen-Kuo-Lin} we could prove that under (\ref{132}%
)-(\ref{45}), the two functions%
\[
y_{\pm a_{1}}\left(  z;\tau \right)  =e^{\pm \frac{z}{2}\left(  \zeta \left(
a_{1}+p\right)  +\zeta \left(  a_{1}-p\right)  \right)  }\frac{\sigma \left(
z\mp a_{1}\right)  }{\left[  \sigma \left(  z+p\right)  \sigma \left(
z-p\right)  \right]  ^{\frac{1}{2}}}%
\]
are two linearly independent solutions to the generalized Lam\'{e} equation
(\ref{152}) with $n_{k}=0$, $k=0,1,2,3$, i.e.,
\begin{equation}
y^{\prime \prime}=\left[  \frac{3}{4}\left(  \wp \left(  z+p\right)  +\wp \left(
z-p\right)  \right)  +A\left(  \zeta \left(  z+p\right)  -\zeta \left(
z-p\right)  \right)  +B\right]  y. \label{187}%
\end{equation}
Observe that (\ref{187}) has singularties only at $\pm p$. Thus, the monodromy
representation of (\ref{187}) is a group homomorphism $\rho:\pi_{1}\left(
E_{\tau}\backslash \left \{  \pm p\right \}  ,q_{0}\right)  \rightarrow SL\left(
2,\mathbb{C}\right)  $. Then we could also compute the monodromy group of
(\ref{187}) with respect to $\left(  y_{a_{1}}\left(  z;\tau \right)
,y_{-a_{1}}\left(  z;\tau \right)  \right)  ^{t}$ as following
\cite{Chen-Kuo-Lin}:%
\begin{equation}
\rho(\gamma_{\pm})\left(
\begin{matrix}
y_{a_{1}}\left(  z;\tau \right) \\
y_{-a_{1}}\left(  z;\tau \right)
\end{matrix}
\right)  =\left(
\begin{matrix}
-1 & 0\\
0 & -1
\end{matrix}
\right)  \left(
\begin{matrix}
y_{a_{1}}\left(  z;\tau \right) \\
y_{-a_{1}}\left(  z;\tau \right)
\end{matrix}
\right)  , \label{155-1}%
\end{equation}%
\begin{equation}
\rho(\ell_{1})\left(
\begin{matrix}
y_{a_{1}}\left(  z;\tau \right) \\
y_{-a_{1}}\left(  z;\tau \right)
\end{matrix}
\right)  =\left(
\begin{matrix}
e^{-2\pi is} & 0\\
0 & e^{2\pi is}%
\end{matrix}
\right)  \left(
\begin{matrix}
y_{a_{1}}\left(  z;\tau \right) \\
y_{-a_{1}}\left(  z;\tau \right)
\end{matrix}
\right)  , \label{155}%
\end{equation}%
\begin{equation}
\rho(\ell_{2})\left(
\begin{matrix}
y_{a_{1}}\left(  z;\tau \right) \\
y_{-a_{1}}\left(  z;\tau \right)
\end{matrix}
\right)  =\left(
\begin{matrix}
e^{2\pi ir} & 0\\
0 & e^{-2\pi ir}%
\end{matrix}
\right)  \left(
\begin{matrix}
y_{a_{1}}\left(  z;\tau \right) \\
y_{-a_{1}}\left(  z;\tau \right)
\end{matrix}
\right)  . \label{156}%
\end{equation}
By (\ref{45}) and Lemma \ref{lem-apparent}, $\pm p\left(  \tau \right)  $ are
apparent singularities. Since the pair $\left(  r,s\right)  $ is fixed, we see
from (\ref{155-1})-(\ref{156}) that the generalized Lam\'{e} equation
(\ref{187}) is monodromy preserving. Thus $Y=\left(
\begin{matrix}
y_{a_{1}}\left(  z;\tau \right)  & y_{-a_{1}}\left(  z;\tau \right) \\
y_{a_{1}}^{\prime}\left(  z;\tau \right)  & y_{-a_{1}}^{\prime}\left(
z;\tau \right)
\end{matrix}
\right)  $ is an $M$-invariant fundamental solution for the system
(\ref{153}). Then by Theorem \ref{thm M}, the single-valued matrix
$\Omega \left(  z;\tau \right)  $ could be defined by%
\begin{align*}
\Omega \left(  z;\tau \right)   &  =\frac{\partial}{\partial \tau}Y\cdot Y^{-1}\\
&  =\frac{1}{\det Y}\left(
\begin{matrix}
\frac{\partial}{\partial \tau}y_{a_{1}} & \frac{\partial}{\partial \tau
}y_{-a_{1}}\\
\frac{\partial}{\partial \tau}y_{a_{1}}^{\prime} & \frac{\partial}{\partial
\tau}y_{-a_{1}}^{\prime}%
\end{matrix}
\right)  \left(
\begin{matrix}
y_{-a_{1}}^{\prime} & -y_{-a_{1}}\\
-y_{a_{1}}^{\prime} & y_{a_{1}}%
\end{matrix}
\right)  ,
\end{align*}
which gives us {\allowdisplaybreaks%
\[
\Omega_{12}=\frac{y_{a_{1}}\frac{\partial}{\partial \tau}y_{-a_{1}}-y_{-a_{1}%
}\frac{\partial}{\partial \tau}y_{a_{1}}}{y_{a_{1}}y_{-a_{1}}^{\prime
}-y_{-a_{1}}y_{a_{1}}^{\prime}}=\frac{\frac{\partial}{\partial \tau}\ln
\frac{y_{a_{1}}}{y_{-a_{1}}}}{\frac{\partial}{\partial z}\ln \frac{y_{a_{1}}%
}{y_{-a_{1}}}}=\frac{\frac{\partial}{\partial \tau}\ln f\left(  z;\tau \right)
}{\frac{\partial}{\partial z}\ln f\left(  z;\tau \right)  },
\]
}where $f\doteqdot \frac{y_{a_{1}}}{y_{-a_{1}}}$ is given by%
\[
f\left(  z;\tau \right)  =e^{z\left(  \zeta \left(  a_{1}+p\right)
+\zeta \left(  a_{1}-p\right)  \right)  }\frac{\sigma \left(  z-a_{1}\right)
}{\sigma \left(  z+a_{1}\right)  }.
\]
Using (\ref{132}) and Legendre relation $\tau \eta_{1}-\eta_{2}=2\pi i$, we
have%
\begin{equation}
f\left(  z;\tau \right)  =e^{2za_{1}\eta_{1}-4\pi isz}\frac{\sigma \left(
z-a_{1}\right)  }{\sigma \left(  z+a_{1}\right)  }. \label{172}%
\end{equation}
In order to compute $\Omega_{12}$, we compute $\frac{\partial}{\partial \tau
}\ln f\left(  z;\tau \right)  $ and $\frac{\partial}{\partial z}\ln f\left(
z;\tau \right)  $, respectively. By Lemma \ref{lem4-1} and (\ref{172}), we
have{\allowdisplaybreaks%
\begin{align}
&  \frac{\partial}{\partial \tau}\ln f\left(  z;\tau \right) \label{173}\\
=  &  2zs\eta_{1}+2za_{1}\frac{d\eta_{1}}{d\tau}-\left(  \zeta \left(
z-a_{1}|\tau \right)  +\zeta \left(  z+a_{1}|\tau \right)  \right)  s\nonumber \\
&  +\frac{\partial}{\partial \tau}\ln \sigma \left(  z-a_{1}|\tau \right)
-\frac{\partial}{\partial \tau}\ln \sigma \left(  z+a_{1}|\tau \right) \nonumber \\
=  &  \frac{i}{4\pi}\left[  \zeta \left(  z+a_{1}\right)  -\zeta \left(
z-a_{1}\right)  -2\eta_{1}a_{1}+4\pi is\right] \nonumber \\
&  \times \left[  \zeta \left(  z-a_{1}\right)  +\zeta \left(  z+a_{1}\right)
-2z\eta_{1}\right]  +\frac{i}{4\pi}\left[  \wp \left(  z-a_{1}\right)
-\wp \left(  z+a_{1}\right)  \right]  ,\nonumber
\end{align}
}and
\begin{equation}
\frac{\partial}{\partial z}\ln f\left(  z;\tau \right)  =2a_{1}\eta_{1}-4\pi
is+\zeta \left(  z-a_{1}\right)  -\zeta \left(  z+a_{1}\right)  . \label{174}%
\end{equation}
\medskip Thus from (\ref{173}), (\ref{174}) and (\ref{132}), we
have{\allowdisplaybreaks%
\begin{align}
\Omega_{12}\left(  z;\tau \right)  =  &  -\frac{i}{4\pi}\left[  \zeta \left(
z-a_{1}\right)  +\zeta \left(  z+a_{1}\right)  -2z\eta_{1}\right] \nonumber \\
&  +\frac{i}{4\pi}\frac{\wp \left(  z-a_{1}\right)  -\wp \left(  z+a_{1}\right)
}{2a_{1}\eta_{1}-4\pi is+\zeta \left(  z-a_{1}\right)  -\zeta \left(
z+a_{1}\right)  }\nonumber \\
=  &  -\frac{i}{4\pi}\left[  \zeta \left(  z-a_{1}\right)  +\zeta \left(
z+a_{1}\right)  -2z\eta_{1}\right] \label{175}\\
&  +\frac{i}{4\pi}\frac{\wp \left(  z-a_{1}\right)  -\wp \left(  z+a_{1}\right)
}{\zeta \left(  a_{1}+p\right)  +\zeta \left(  a_{1}-p\right)  +\zeta \left(
z-a_{1}\right)  -\zeta \left(  z+a_{1}\right)  }.\nonumber
\end{align}
}From (\ref{175}), we see that $\pm a_{1}$ are not poles of $\Omega
_{12}\left(  z;\tau \right)  $. In fact, $\pm p$ are the only simple poles and
$0$ is a zero of $\Omega_{12}\left(  z;\tau \right)  $. Furthermore, we have%
\begin{equation}
\underset{z=\pm p}{\text{Res}}\Omega_{12}\left(  z;\tau \right)  =\frac
{-i}{4\pi}. \label{157}%
\end{equation}
By (\ref{175}), it is easy to see that%
\begin{equation}
\Omega_{12}\left(  -z;\tau \right)  =-\Omega_{12}\left(  z;\tau \right)  .
\label{188}%
\end{equation}
By (\ref{157}), (\ref{188}) and (\ref{qqq}), $\Omega_{12}\left(
z;\tau \right)  $ has a simpler expression as follows:%
\[
\Omega_{12}\left(  z;\tau \right)  =-\frac{i}{4\pi}\left(  \zeta \left(
z-p\right)  +\zeta \left(  z+p\right)  -2z\eta_{1}\right)  .
\]

For the general case, we do not have the explicit expression of the two
linearly independent solutions. But the discussion above motivates us to find
the explicit form of $\Omega_{12}$. For example, we might ask whether there
exists $\Omega_{12}$ satisfying the property (\ref{188}) or not. Thus, we need
to study it via a different way. More precisely, we prove the following
theorem:\medskip

\begin{theorem}
\label{thmA}Under the assumption of Theorem \ref{thm4-2}, suppose the
generalized Lam\'{e} equation (\ref{152}) with $\left(  p,A\right)  =\left(
p\left(  \tau \right)  ,A\left(  \tau \right)  \right)  $ is an isomonodromic
deformation with respect to $\tau$. Then there exists an $M$-invariant
fundamental solution $Y\left(  z;\tau \right)  $ of system (\ref{153}) such
that $\Omega_{12}\left(  z;\tau \right)  $ is of the form:%
\begin{equation}
\Omega_{12}\left(  z;\tau \right)  =-\frac{i}{4\pi}\left(  \zeta \left(
z-p\left(  \tau \right)  \right)  +\zeta \left(  z+p\left(  \tau \right)
\right)  -2z\eta_{1}\right)  , \label{c}%
\end{equation}
where $\Omega_{12}\left(  z;\tau \right)  $ is the (1,2) component of
$\Omega \left(  z;\tau \right)  $ which is defined by $Y\left(  z;\tau \right)  $.
\end{theorem}

We remark that Theorem \ref{thmA} is a result locally in $\tau$. In the
following, we always assume that $V_{0}$ is a small neighborhood of $\tau_{0}$
such that $p\left(  \tau \right)  \not \in E_{\tau}\left[  2\right]  $ and
$A\left(  \tau \right)  $, $B\left(  \tau \right)  $ are finite for $\tau \in
V_{0}$. First, we study the singularities of $\Omega_{12}\left(
z;\tau \right)  $:

\begin{lemma}
\label{lem4-5}Under the assumption and notations of Theorem \ref{thm4-2},
suppose $Y\left(  z;\tau \right)  $ is an $M$-invariant fundamental solution of
(\ref{153}) with $\left(  p,A\right)  =\left(  p\left(  \tau \right)  ,A\left(
\tau \right)  \right)  $ and $\Omega \left(  z;\tau \right)  $ is defined by
$Y\left(  z;\tau \right)  $. Then

\begin{itemize}
\item[(i)] $\Omega_{12}\left(  \cdot;\tau \right)  $ is meromorphic in
$\mathbb{C}$ and holomorphic for all $z\not \in \{ \pm p(\tau),\frac
{\omega_{i}}{2}$, $i=0,1,2,3\}+\Lambda_{\tau}$.

\item[(ii)] If there exist $i\in \{0,1,2,3\}$ and $\left(  b_{1},b_{2}\right)
\in \mathbb{Z}^{2}$ such that $\frac{\omega_{i}}{2}+b_{1}+b_{2}\tau$ is a pole
of $\Omega_{12}\left(  \cdot;\tau \right)  $ with order $m_{i}$, then
$m_{i}=2n_{i}$, and any point in $\frac{\omega_{i}}{2}+\Lambda_{\tau}$ is also
a pole of $\Omega_{12}\left(  \cdot;\tau \right)  $ with the same order $m_{i}%
$. Consequently, if $\Omega_{12}\left(  \cdot;\tau \right)  $ has a pole at
$\frac{\omega_{i}}{2}+\Lambda_{\tau}$, then $n_{i}\in \mathbb{N}$.

\item[(iii)] $\Omega_{12}\left(  \cdot;\tau \right)  $ has poles at $\left \{
\pm p\right \}  +\Lambda_{\tau}$ of order at most one.
\end{itemize}
\end{lemma}

\begin{proof}
(i) Since equation (\ref{186}) has singularities only at $\{ \pm p\left(
\tau \right)  ,\frac{\omega_{i}}{2},i=0,1,2,3\}+\Lambda_{\tau}$, $\Omega
_{12}\left(  \cdot;\tau \right)  $ is holomorphic for all $z\not \in \{ \pm
p(\tau),\frac{\omega_{i}}{2},i=0,1,2,3\}+\Lambda_{\tau}$. On the other hand,
if $z_{0}\in \{ \pm p\left(  \tau \right)  ,\frac{\omega_{i}}{2}%
,i=0,1,2,3\}+\Lambda_{\tau}$ is a singularity of $\Omega_{12}\left(
\cdot;\tau \right)  $, then by using (\ref{w}) and the local behavior of
$Y\left(  \cdot;\tau \right)  $ at $z_{0}$, it is easy to prove%
\[
\Omega_{12}\left(  z;\tau \right)  =\frac{c(\tau)}{(z-z_{0})^{m}}%
(1+\text{higher order term})\text{ near }z_{0}%
\]
for some $c(\tau)\not =0$ and $m\in \mathbb{C}$. Since $\Omega_{12}\left(
\cdot;\tau \right)  $ is single-valued, we conclude that $m\in \mathbb{N}$,
namely $z_{0}$ must be a pole of $\Omega_{12}\left(  \cdot;\tau \right)  $.
This proves (i).

The proof of (ii) and (iii) are similar, so we only prove (ii) for $i=0$.
Without loss of generality, we may assume $b_{1}=b_{2}=0$. Suppose $0$ is a
pole of $\Omega_{12}\left(  z;\tau \right)  $ with order $m_{0}\in \mathbb{N}$.
By (\ref{qqq}), it is obvious that for any $\left(  b_{1},b_{2}\right)
\in \mathbb{Z}^{2}$, $b_{1}+b_{2}\tau$ is also a pole with the same order
$m_{0}$. Suppose%
\begin{equation}
\Omega_{12}\left(  z;\tau \right)  =z^{-m_{0}}\left(  \sum_{k=0}^{\infty}%
c_{k}z^{k}\right)  =\frac{c_{0}}{z^{m_{0}}}+O\left(  \frac{1}{z^{m_{0}-1}%
}\right)  \text{ near }0, \label{190}%
\end{equation}
where $c_{0}\neq0$. Then we have%
\begin{equation}
\Omega_{12}^{\prime}(z;\tau)=-m_{0}\frac{c_{0}}{z^{m_{0}+1}}+O\left(  \frac
{1}{z^{m_{0}}}\right)  , \label{192}%
\end{equation}%
\begin{equation}
\Omega_{12}^{\prime \prime \prime}\left(  z;\tau \right)  =-m_{0}\left(
m_{0}+1\right)  \left(  m_{0}+2\right)  \frac{c_{0}}{z^{m_{0}+3}}+O\left(
\frac{1}{z^{m_{0}+2}}\right)  , \label{191}%
\end{equation}
and{\allowdisplaybreaks%
\begin{align}
I\left(  z;\tau \right)   &  =n_{0}\left(  n_{0}+1\right)  \frac{1}{z^{2}%
}\nonumber \\
&  +\left[  \sum_{i=1}^{3}n_{i}\left(  n_{i}+1\right)  \wp \left(  \frac
{\omega_{i}}{2}\right)  +\frac{3}{2}\wp \left(  p\right)  +2A\zeta \left(
p\right)  +B\right]  +O\left(  z\right) \nonumber \\
&  =n_{0}\left(  n_{0}+1\right)  \frac{1}{z^{2}}+D\left(  \tau \right)
+O\left(  z\right)  , \label{193}%
\end{align}
}where $D\left(  \tau \right)  $ is a constant depending on $\tau$. Thus%
\begin{equation}
I^{\prime}(z;\tau)=-2n_{0}\left(  n_{0}+1\right)  \frac{1}{z^{3}}+O\left(
1\right)  ,\text{ \  \ }\frac{\partial I}{\partial \tau}(z;\tau)=O\left(
1\right)  . \label{194}%
\end{equation}
Substituting (\ref{190})-(\ref{194}) into (\ref{186}), we easily obtain%
\[
m_{0}\left(  m_{0}+1\right)  \left(  m_{0}+2\right)  c_{0}=4n_{0}\left(
n_{0}+1\right)  \left(  m_{0}+1\right)  c_{0}.
\]
Since $n_{0}\geq-\frac{1}{2}$, we have $m_{0}=2n_{0}\in \mathbb{N}$. Together
with the assumption that $n_{0}\not \in \frac{1}{2}+\mathbb{Z}$, we have
$n_{0}\in \mathbb{N}$. This completes the proof.
\end{proof}

For the isomonodromic deformation of the 2nd order Fuchsian equation
(\ref{90}) on $\mathbb{CP}^{1}$, if the non-resonant condition $n_{i}%
\not \in \frac{1}{2}+\mathbb{Z}$ holds, then $\Omega_{12}$ is independent of
the choice of $M$-invariant fundamental solutions. See \cite{GP}. However, the
same conclusion is not true in our study of equations defined in tori; see
Remark \ref{rmk} below. The following lemma is to classify the structure of
solutions of (\ref{186}).

\begin{lemma}
\label{lemI}Under the assumption and notations of Lemma \ref{lem4-5}. Then

\begin{itemize}
\item[(i)] If $\tilde{Y}\left(  z;\tau \right)  $ is another $M$-invariant
fundamental solution of (\ref{153}), then $\Omega_{12}\left(  z;\tau \right)
-\tilde{\Omega}_{12}\left(  z;\tau \right)  $ is an elliptic function with
periods $1$ and $\tau$, and satisfies the following second symmetric product
equation of (\ref{152}):%
\begin{equation}
\Phi^{\prime \prime \prime}-4I\Phi^{\prime}-2I^{\prime}\Phi=0. \label{189}%
\end{equation}

\item[(ii)] Let $\Phi(z;\tau)$ be an elliptic solution of (\ref{189}). For any
$c\in \mathbb{C}$, define $\tilde{\Omega}_{12}\left(  z;\tau \right)  $ by%
\[
\tilde{\Omega}_{12}\left(  z;\tau \right)  \doteqdot \Omega_{12}\left(
z;\tau \right)  +c\Phi \left(  z;\tau \right)  .
\]
Then there exists an $M$-invariant fundamental solution $\tilde{Y}\left(
z;\tau \right)  $ of system (\ref{153}) such that $\tilde{\Omega}_{12}\left(
z;\tau \right)  $ is the (1,2) component of $\tilde{\Omega}\left(
z;\tau \right)  $ which is defined by $\tilde{Y}\left(  z;\tau \right)  $.
\end{itemize}
\end{lemma}

\begin{proof}
(i) This follows directly from that $\Omega_{12}\left(  z;\tau \right)  $ and
$\tilde{\Omega}_{12}\left(  z;\tau \right)  $ are both single-valued and
satisfy (\ref{186}) and (\ref{qqq}).

(ii) It is trivial to see that $\tilde{\Omega}_{12}\left(  z;\tau \right)  $
satisfies (\ref{186}) and (\ref{qqq}). Moreover, since both $\Phi \left(
z;\tau \right)  $ and $\Omega_{12}\left(  z;\tau \right)  $ are single-valued,
$\tilde{\Omega}_{12}\left(  z;\tau \right)  $ is single-valued. By Theorem
\ref{thm M1}, there exists an $M$-invariant fundamental solution $\tilde
{Y}\left(  z;\tau \right)  $ of system (\ref{153}) such that $\tilde{\Omega
}\left(  z;\tau \right)  $ is defined by $\tilde{Y}\left(  z;\tau \right)  $.
\end{proof}

\begin{lemma}
\label{lem4-4}Under the assumption and notations of Lemma \ref{lem4-5}. Then
there exists an $M$-invariant fundamental solution $\tilde{Y}\left(
z;\tau \right)  $ such that
\[
\tilde{\Omega}_{12}\left(  z;\tau \right)  =-\Omega_{12}\left(  -z;\tau \right)
.
\]

\end{lemma}

\begin{proof}
Recall that $Y\left(  z;\tau \right)  =\left(
\begin{matrix}
y_{1}\left(  z;\tau \right)  & y_{2}\left(  z;\tau \right) \\
y_{1}^{\prime}\left(  z;\tau \right)  & y_{2}^{\prime}\left(  z;\tau \right)
\end{matrix}
\right)  $ is an $M$-invariant fundamental solution of (\ref{153}) in a
neighborhood $U_{q_{0}}$ of $q_{0}$. Then for $z\in-U_{q_{0}}$, a neighborhood
of $-q_{0}$, we define
\[
\tilde{Y}\left(  z;\tau \right)  :=\left(
\begin{matrix}
y_{1}\left(  -z;\tau \right)  & y_{2}\left(  -z;\tau \right) \\
-y_{1}^{\prime}\left(  -z;\tau \right)  & -y_{2}^{\prime}\left(  -z;\tau
\right)
\end{matrix}
\right)  .
\]
It is easy to see that $\tilde{Y}\left(  z;\tau \right)  $ is a fundamental
solution to (\ref{153}) in $-U_{q_{0}}$. Define $\tilde{\Omega}\left(
z;\tau \right)  $ by $\tilde{Y}\left(  z;\tau \right)  $, then we have%
\[
\det \tilde{Y}\left(  z;\tau \right)  \cdot \tilde{\Omega}_{12}\left(
z;\tau \right)  =y_{1}\left(  -z;\tau \right)  \frac{\partial}{\partial \tau
}y_{2}\left(  -z;\tau \right)  -y_{2}\left(  -z;\tau \right)  \frac{\partial
}{\partial \tau}y_{1}\left(  -z;\tau \right)  ,
\]
and since $\det \tilde{Y}\left(  z;\tau \right)  =-\det Y\left(  -z;\tau \right)
$, we obtain%
\begin{equation}
\tilde{\Omega}_{12}\left(  z;\tau \right)  =-\Omega_{12}\left(  -z;\tau \right)
\label{181}%
\end{equation}
for $z\in-U_{q_{0}}$. Since $\Omega_{12}$ is globally defined and
single-valued, by analytic continuation, (\ref{181}) holds true globally.
Thus, $\tilde{\Omega}_{12}$ is globally defined and single-valued. Moreover,
$\tilde{\Omega}_{12}\left(  z;\tau \right)  $ satisfies (\ref{186}) and
(\ref{qqq}) which implies that $\tilde{Y}\left(  z;\tau \right)  $ is
$M$-invariant. This completes the proof.
\end{proof}

\begin{proof}
[Proof of Theorem \ref{thmA}]Since the generalized Lam\'{e} equation
(\ref{152}) with (\ref{154}) is monodromy preserving as $\tau$ deforms, by
Theorem \ref{thm M1} and Lemma \ref{lem4-5}, there exists a single-valued
meromorphic function $\hat{\Omega}_{12}\left(  z;\tau \right)  $ satisfying
(\ref{186}) and (\ref{qqq}). Define $\Omega_{12}\left(  z;\tau \right)  $ by
\begin{equation}
\Omega_{12}\left(  z;\tau \right)  \doteqdot \frac{1}{2}\left[  \hat{\Omega
}_{12}\left(  z;\tau \right)  -\hat{\Omega}_{12}\left(  -z;\tau \right)
\right]  . \label{206}%
\end{equation}
To prove Theorem \ref{thmA}, we divide it into three steps:

\textbf{Step 1.} We prove that there exists an $M$-invariant fundamental
solution $Y\left(  z;\tau \right)  $ of system (\ref{153}) such that\emph{ }%
\begin{equation}
\Omega \left(  z;\tau \right)  =\frac{\partial}{\partial \tau}Y\left(
z;\tau \right)  \cdot Y^{-1}\left(  z;\tau \right)  \label{c3}%
\end{equation}
and $\Omega_{12}\left(  z;\tau \right)  $ is the (1,2) component of
$\Omega \left(  z;\tau \right)  $.

Let
\[
\Phi(z;\tau)=-\frac{1}{2}\left[  \hat{\Omega}_{12}\left(  z;\tau \right)
+\hat{\Omega}_{12}\left(  -z;\tau \right)  \right]  .
\]
By Lemmas \ref{lem4-4} and \ref{lemI}, $\Phi$ is an elliptic solution of
equation (\ref{189}) and%
\[
\Omega_{12}\left(  z;\tau \right)  =\hat{\Omega}_{12}\left(  z;\tau \right)
+\Phi(z;\tau)\text{.}%
\]
By Lemma \ref{lemI} (ii), there exists an $M$-invariant fundamental solution
$Y\left(  z;\tau \right)  $ of system (\ref{153}) such that (\ref{c3}) holds.

\textbf{Step 2. }We prove that $\Omega_{12}\left(  z;\tau \right)  $ is an odd
meromorphic function and only has poles at $\left \{  \pm p\right \}
+\Lambda_{\tau}$ of order at most one. Furthermore, $\Omega_{12}^{\prime
}\left(  z;\tau \right)  $ is an even elliptic function.

Clearly (\ref{206}) and Lemma \ref{lem4-5} imply that $\Omega_{12}\left(
z;\tau \right)  $ is an odd meromorphic function. Now we claim that:%
\begin{equation}
\Omega_{12}\left(  z;\tau \right)  \text{ only has poles at}\left \{  \pm
p\right \}  +\Lambda_{\tau}\text{ of order at most one.} \label{213}%
\end{equation}
By Lemma \ref{lem4-5} (i), $\Omega_{12}\left(  z;\tau \right)  $ is holomorphic
for all $z\not \in \{ \pm p,\frac{\omega_{i}}{2},i=0,1,2,3\}+\Lambda_{\tau}$.
If $\Omega_{12}\left(  z;\tau \right)  $ has a pole at $\frac{\omega_{i}}%
{2}+\Lambda_{\tau}$, then the order of the pole is $2n_{i}\in2\mathbb{N}$ by
Lemma \ref{lem4-5} (ii), which yields a contradiction to the fact that
$\Omega_{12}\left(  z;\tau \right)  $ is odd and satisfies (\ref{qqq}).

\textbf{Step 3. }We prove that $\Omega_{12}\left(  z;\tau \right)  $ is of the
form (\ref{c}):%
\[
\Omega_{12}\left(  z;\tau \right)  =-\frac{i}{4\pi}\left(  \zeta \left(
z-p\right)  +\zeta \left(  z+p\right)  -2z\eta_{1}\right)  .
\]

By Step 2\textbf{ }and (\ref{213}), we know that $\Omega_{12}^{\prime}\left(
z;\tau \right)  $ must be of the following form%
\[
\Omega_{12}^{\prime}\left(  z;\tau \right)  =-C\left(  \wp \left(  z+p\right)
+\wp \left(  z-p\right)  \right)  +D
\]
for some constants $C,D\in \mathbb{C}$. Thus by integration, we get
\[
\Omega_{12}\left(  z;\tau \right)  =C\left(  \zeta \left(  z+p\right)
+\zeta \left(  z-p\right)  \right)  +Dz+E
\]
for some $E\in \mathbb{C}$. Since $\Omega_{12}\left(  z;\tau \right)  $ is odd,
we have $E=0$. Furthermore,%
\[
\Omega_{12}\left(  z+1;\tau \right)  =\Omega_{12}\left(  z;\tau \right)
+2C\eta_{1}+D,
\]
and%
\[
\Omega_{12}\left(  z+\tau;\tau \right)  =\Omega_{12}\left(  z;\tau \right)
+2C\eta_{2}+D\tau.
\]
By (\ref{204}), we have%
\[
2C\eta_{1}+D=0,\text{ }2C\eta_{2}+D\tau=-1.
\]
By Legendre relation $\tau \eta_{1}-\eta_{2}=2\pi i$, we have%
\[
C=\frac{-i}{4\pi}\text{ \ and \ }D=\frac{i}{2\pi}\eta_{1},
\]
which implies (\ref{c}). This completes the proof.
\end{proof}

\begin{corollary}
Under the assumption and notations of Lemma \ref{lem4-5} and assume
$n_{i}\not \in \mathbb{Z}$ for some $i\in \{0,1,2,3\}$. Then $\Omega
_{12}\left(  z;\tau \right)  $ is unique, i.e., $\Omega_{12}\left(
z;\tau \right)  $ is independent of the choice of $M$-invariant solution
$Y\left(  z;\tau \right)  $ of system (\ref{153}).
\end{corollary}

\begin{proof}
For any $M$-invariant solution $Y\left(  z;\tau \right)  $ of system
(\ref{153}), by Theorem \ref{thm M1}, there exists a single-valued function
$\Omega_{12}\left(  z;\tau \right)  $ satisfying (\ref{186}) and (\ref{204}).
Let%
\[
\Phi \left(  z;\tau \right)  =\Omega_{12}\left(  z;\tau \right)  +\Omega
_{12}\left(  -z;\tau \right)  .
\]
If $\Phi \left(  z;\tau \right)  \not \equiv 0$, then $\Phi \left(
z;\tau \right)  $ is an even elliptic solution of (\ref{189}). Without loss of
generality, we may consider the case $n_{1}\not \in \mathbb{Z}$. Then
$2n_{1}\not \in \mathbb{Z}$ since $n_{1}\not \in \frac{1}{2}+\mathbb{Z}$.
Since the local exponents of (\ref{189}) at $\frac{\omega_{1}}{2}$ are
$-2n_{1},1,2n_{1}+2$ and $\Phi \left(  z;\tau \right)  $ is elliptic, the local
exponent of $\Phi \left(  z;\tau \right)  $ at $z=\frac{\omega_{1}}{2}$ must be
$1$, i.e., $\frac{\omega_{1}}{2}$ is a simple zero. But again by $\Phi \left(
z;\tau \right)  $ is even elliptic, we have $\Phi^{\prime}\left(  \frac
{\omega_{1}}{2};\tau \right)  =0$, which leads to a contradiction. Thus,
$\Phi \left(  z;\tau \right)  \equiv0$, i.e., $\Omega_{12}\left(  z;\tau \right)
$ is odd. Then by Theorem \ref{thmA}, $\Omega_{12}\left(  z;\tau \right)  $ is
of the form (\ref{c}).
\end{proof}

\begin{remark}
\label{rmk}When $n_{i}\in \mathbb{Z}$ for all $i=0,1,2,3$, $\Omega_{12}\left(
z;\tau \right)  $ might not be unique. For example, when $n_{i}=0$ for all
$i=0,1,2,3$, we define%
\[
\Phi \left(  z;\tau \right)  \doteqdot \zeta \left(  z+p|\tau \right)
-\zeta \left(  z-p|\tau \right)  -\zeta \left(  2p|\tau \right)  -2A,
\]
then $\Phi \left(  z;\tau \right)  $ is an even elliptic solution of
(\ref{189}). So for any $c\in \mathbb{C}$,%
\[
\tilde{\Omega}_{12}\left(  z;\tau \right)  \doteqdot \frac{-i}{4\pi}\left(
\zeta \left(  z-p\right)  +\zeta \left(  z+p\right)  -2z\eta_{1}\right)
+c\Phi \left(  z;\tau \right)
\]
satisfies (\ref{186}) and (\ref{qqq}). By Lemma \ref{lemI}, there exists an
$M$-invariant solution $\tilde{Y}\left(  z;\tau \right)  $ such that
$\tilde{\Omega}\left(  z;\tau \right)  $ is defined by $\tilde{Y}\left(
z;\tau \right)  $.\smallskip
\end{remark}

Define $U\left(  z;\tau \right)  $ by%
\[
U\left(  z;\tau \right)  \doteqdot \Omega_{12}^{\prime \prime \prime}\left(
z;\tau \right)  -4I\left(  z;\tau \right)  \Omega_{12}^{\prime}\left(
z;\tau \right)  -2I^{\prime}\left(  z;\tau \right)  \Omega_{12}\left(
z;\tau \right)  +2\frac{\partial}{\partial \tau}I\left(  z;\tau \right)  ,
\]
where $\Omega_{12}\left(  z;\tau \right)  $ is given in Theorem \ref{thmA}
(\ref{c}), i.e.,%
\[
\Omega_{12}\left(  z;\tau \right)  =-\frac{i}{4\pi}\left(  \zeta \left(
z-p\right)  +\zeta \left(  z+p\right)  -2z\eta_{1}\right)  .
\]
In order to prove Theorem \ref{thm4-2}, we need the following local expansions
for $\Omega_{12}\left(  z;\tau \right)  $ and $I\left(  z;\tau \right)  $ at $p$
and $\frac{\omega_{k}}{2}$, $k=0,1,2,3$, respectively.

\begin{lemma}
\label{lem-expand}$\Omega_{12}\left(  z;\tau \right)  $ and $I\left(
z;\tau \right)  $ have local expansions at $p$ and $\frac{\omega_{k}}{2}$,
$k=0,1,2,3$ as follows:

\begin{itemize}
\item[(i)] Near $p$, let $u=z-p$. Then we have%
\begin{equation}
\Omega_{12}\left(  z;\tau \right)  =\frac{-i}{4\pi}\left(
\begin{array}
[c]{l}%
u^{-1}+\left(  \zeta \left(  2p\right)  -2p\eta_{1}\right)  -\left(  \wp \left(
2p\right)  +2\eta_{1}\right)  u\\
-\frac{1}{2}\wp^{\prime}\left(  2p\right)  u^{2}-\frac{1}{6}\left(
\frac{g_{2}}{10}+\wp^{\prime \prime}\left(  2p\right)  \right)  u^{3}+O\left(
u^{4}\right)
\end{array}
\right)  , \label{390}%
\end{equation}
and%
\begin{equation}
I\left(  z;\tau \right)  =\frac{3}{4}u^{-2}-Au^{-1}+A^{2}+H_{1}\left(
\tau \right)  u+H_{2}\left(  \tau \right)  u^{2}+O\left(  u^{3}\right)  ,
\label{391}%
\end{equation}
where%
\begin{equation}
H_{1}\left(  \tau \right)  =\sum_{k=0}^{3}n_{k}\left(  n_{k}+1\right)
\wp^{\prime}\left(  p+\frac{\omega_{k}}{2}\right)  +\frac{3}{4}\wp^{\prime
}\left(  2p\right)  -A\wp \left(  2p\right)  , \label{404}%
\end{equation}
and%
\[
H_{2}\left(  \tau \right)  =\frac{1}{2}\left[  \sum_{k=0}^{3}n_{k}\left(
n_{k}+1\right)  \wp^{\prime \prime}\left(  p+\frac{\omega_{k}}{2}\right)
+\frac{3}{4}\wp^{\prime \prime}\left(  2p\right)  +\frac{3}{40}g_{2}%
-A\wp^{\prime}\left(  2p\right)  \right]  .
\]

\item[(ii)] Near $\frac{\omega_{k}}{2}$, $k\in \{0,1,2,3\}$, let $u_{k}%
=z-\frac{\omega_{k}}{2}$. Then we have%
\begin{equation}
\Omega_{12}\left(  z;\tau \right)  =\frac{i}{4\pi}\left[
\begin{array}
[c]{l}%
-\left(  \zeta \left(  \frac{\omega_{k}}{2}+p\right)  +\zeta \left(
\frac{\omega_{k}}{2}-p\right)  -\omega_{k}\eta_{1}\right) \\
+2\left(  \wp \left(  \frac{\omega_{k}}{2}+p\right)  +\eta_{1}\right)
u_{k}+O\left(  u_{k}^{3}\right)
\end{array}
\right]  , \label{392}%
\end{equation}
and%
\begin{equation}
I\left(  z;\tau \right)  =n_{k}\left(  n_{k}+1\right)  u_{k}^{-2}+\Lambda
_{k}\left(  \tau \right)  +O\left(  u_{k}^{2}\right)  , \label{393}%
\end{equation}
where{\allowdisplaybreaks%
\begin{align}
\Lambda_{k}\left(  \tau \right)  =  &  \sum_{j\not =k}^{3}n_{j}\left(
n_{j}+1\right)  \wp \left(  \frac{\omega_{k}+\omega_{j}}{2}\right)  +\frac
{3}{2}\wp \left(  \frac{\omega_{k}}{2}+p\right) \label{406}\\
&  +A\left(  \tau \right)  \left(  \zeta \left(  \frac{\omega_{k}}{2}+p\right)
-\zeta \left(  \frac{\omega_{k}}{2}-p\right)  \right)  +B\left(  \tau \right)
.\nonumber
\end{align}
}
\end{itemize}
\end{lemma}

\begin{proof}
Recall the following expansions:%
\begin{equation}
\zeta \left(  u\right)  =\frac{1}{u}-\frac{g_{2}}{60}u^{3}-\frac{g_{3}}%
{140}u^{5}+O\left(  u^{7}\right)  , \label{e1}%
\end{equation}%
\begin{equation}
\wp \left(  u\right)  =\frac{1}{u^{2}}+\frac{g_{2}}{20}u^{2}+\frac{g_{3}}%
{28}u^{4}+O\left(  u^{6}\right)  . \label{e2}%
\end{equation}
The proof follows from a direct computation by using (\ref{e1}) and (\ref{e2}).
\end{proof}

By using Lemma \ref{lem-expand}, we have

\begin{lemma}
\label{lem-U}$U\left(  \cdot;\tau \right)  $ is an even elliptic function and
has poles only at $\pm p$ of order at most $3$. More precisely, $U\left(
z;\tau \right)  $ is expressed as follows:{\allowdisplaybreaks%
\begin{align}
U\left(  z;\tau \right)  =  &  L\left(  \tau \right)  \left(  \wp^{\prime
}\left(  z-p\right)  -\wp^{\prime}\left(  z+p\right)  \right) \label{384}\\
&  +M\left(  \tau \right)  \left(  \wp \left(  z-p\right)  +\wp \left(
z+p\right)  \right) \nonumber \\
&  +N\left(  \tau \right)  \left(  \zeta \left(  z-p\right)  -\zeta \left(
z+p\right)  \right)  +C\left(  \tau \right)  ,\nonumber
\end{align}
}where the coefficients $L\left(  \tau \right)  $, $M\left(  \tau \right)  $,
$N\left(  \tau \right)  $ and $C\left(  \tau \right)  $ are given by%
\begin{equation}
L\left(  \tau \right)  =-\frac{1}{2}\left(  3\frac{dp}{d\tau}+\frac{i}{4\pi
}\left[  6A-3\left(  \zeta \left(  2p\right)  -2p\eta_{1}\right)  \right]
\right)  , \label{398}%
\end{equation}%
\begin{equation}
M\left(  \tau \right)  =-2A\frac{dp}{d\tau}+\frac{i}{4\pi}\left[
-4A^{2}+2A\left(  \zeta \left(  2p\right)  -2p\eta_{1}\right)  \right]  ,
\label{399}%
\end{equation}%
\begin{equation}
N\left(  \tau \right)  =-2\frac{dA}{d\tau}+\frac{i}{4\pi}\left[
\begin{array}
[c]{l}%
4A\left(  \wp \left(  2p\right)  +\eta_{1}\right)  -3\wp^{\prime}\left(
2p\right) \\
-2\sum_{k=0}^{3}n_{k}\left(  n_{k}+1\right)  \wp^{\prime}\left(
p+\frac{\omega_{k}}{2}\right)
\end{array}
\right]  , \label{401}%
\end{equation}
{\allowdisplaybreaks%
\begin{align}
C\left(  \tau \right)  =  &  4A\frac{dA}{d\tau}-2H_{1}\left(  \tau \right)
\frac{dp}{d\tau}\label{402}\\
&  +\frac{i}{4\pi}\left[
\begin{array}
[c]{l}%
-4A^{2}\left(  \wp \left(  2p\right)  +2\eta_{1}\right)  +3A\wp^{\prime}\left(
2p\right) \\
+2H_{1}\left(  \tau \right)  \left(  \zeta \left(  2p\right)  -2p\eta
_{1}\right)
\end{array}
\right]  .\nonumber
\end{align}
}Here $H_{1}\left(  \tau \right)  $ is given in (\ref{404}).
\end{lemma}

\begin{proof}
Since $I\left(  \cdot;\tau \right)  $ is elliptic, we have $I\left(
z;\tau \right)  =I\left(  z+\tau;\tau \right)  $. Thus%
\begin{align}
\frac{\partial}{\partial \tau}I\left(  z;\tau \right)   &  =I^{\prime}\left(
z+\tau;\tau \right)  +\frac{\partial}{\partial \tau}I\left(  z+\tau;\tau \right)
\label{382}\\
&  =I^{\prime}\left(  z;\tau \right)  +\frac{\partial}{\partial \tau}I\left(
z+\tau;\tau \right)  .\nonumber
\end{align}
By using (\ref{382}) and the translation property (\ref{qqq}) of $\Omega
_{12}\left(  z;\tau \right)  $, we have%
\[
U\left(  z+\omega_{k};\tau \right)  =U\left(  z;\tau \right)  ,\text{ }k=1,2,
\]
that is, $U\left(  \cdot;\tau \right)  $ is elliptic. Moreover, since $I\left(
\cdot;\tau \right)  $ is even, we have%
\begin{equation}
\frac{\partial}{\partial \tau}I\left(  z;\tau \right)  =\frac{\partial}%
{\partial \tau}I\left(  -z;\tau \right)  . \label{383}%
\end{equation}
By using (\ref{383}) and $\Omega_{12}\left(  \cdot;\tau \right)  $ is odd, we
see that $U\left(  \cdot;\tau \right)  $ is even.\smallskip

Next, we claim that:
\begin{equation}
U\left(  \cdot;\tau \right)  \text{ is holomorphic at }\frac{\omega_{k}}%
{2},\text{ }k=0,1,2,3. \label{408}%
\end{equation}
Since the proof is similar, we only give the proof for $k=2$. In this case, by
(\ref{392}) and (\ref{393}) in Lemma \ref{lem-expand}, near $\frac{\tau}{2}$,
we have,
\begin{equation}
\Omega_{12}\left(  z;\tau \right)  =\frac{i}{4\pi}\left[  2\pi i+2\left(
\wp \left(  \frac{\tau}{2}+p\right)  +\eta_{1}\right)  u_{2}+O\left(  u_{2}%
^{3}\right)  \right]  , \label{385}%
\end{equation}
and%
\begin{equation}
I\left(  z;\tau \right)  =n_{2}\left(  n_{2}+1\right)  u_{2}^{-2}+\Lambda
_{2}\left(  \tau \right)  +O\left(  u_{2}^{2}\right)  , \label{386}%
\end{equation}
Then near $\frac{\tau}{2}$, we have%
\begin{equation}
\Omega_{12}^{\prime}\left(  z;\tau \right)  =\frac{i}{4\pi}\left[  2\left(
\wp \left(  \frac{\tau}{2}+p\right)  +\eta_{1}\right)  +O\left(  u_{2}%
^{2}\right)  \right]  , \label{387}%
\end{equation}%
\begin{equation}
I^{\prime}\left(  z;\tau \right)  =-2n_{2}\left(  n_{2}+1\right)  u_{2}%
^{-3}+O\left(  u_{2}\right)  , \label{388}%
\end{equation}
and%
\begin{equation}
\frac{\partial}{\partial \tau}I\left(  z;\tau \right)  =n_{2}\left(
n_{2}+1\right)  u_{2}^{-3}+\frac{\partial}{\partial \tau}\Lambda_{2}\left(
\tau \right)  +O\left(  u_{2}\right)  . \label{389}%
\end{equation}
By (\ref{385})-(\ref{389}) and $\Omega_{12}\left(  z;\tau \right)  $ is
holomorphic at $\frac{\omega_{k}}{2},$ $k=0,1,2,3$, we
have{\allowdisplaybreaks%
\begin{align}
&  U\left(  z;\tau \right) \label{407}\\
=  &  \Omega_{12}^{\prime \prime \prime}\left(  z;\tau \right)  -4I\left(
z;\tau \right)  \Omega_{12}^{\prime}\left(  z;\tau \right)  -2I^{\prime}\left(
z;\tau \right)  \Omega_{12}\left(  z;\tau \right)  +2\frac{\partial}%
{\partial \tau}I\left(  z;\tau \right) \nonumber \\
=  &  \Omega_{12}^{\prime \prime \prime}\left(  z;\tau \right)  -4\left[
n_{2}\left(  n_{2}+1\right)  u_{2}^{-2}+\Lambda_{2}\left(  \tau \right)
+O\left(  u_{2}^{2}\right)  \right] \nonumber \\
&  \times \left(  \frac{i}{4\pi}\right)  \left[  2\left(  \wp \left(  \frac
{\tau}{2}+p\right)  +\eta_{1}\right)  +O\left(  u_{2}^{2}\right)  \right]
\nonumber \\
&  -2\left[  -2n_{2}\left(  n_{2}+1\right)  u_{2}^{-3}+O\left(  u_{2}\right)
\right] \nonumber \\
&  \times \left(  \frac{i}{4\pi}\right)  \left[  2\pi i+2\left(  \wp \left(
\frac{\tau}{2}+p\right)  +\eta_{1}\right)  u_{1}+O\left(  u_{2}^{3}\right)
\right] \nonumber \\
&  +2\left[  n_{2}\left(  n_{2}+1\right)  u_{2}^{-3}+\frac{\partial}%
{\partial \tau}\Lambda_{2}\left(  \tau \right)  +O\left(  u_{2}^{2}\right)
\right]  .\nonumber
\end{align}
}From (\ref{407}), it is easy to see that the coefficients of $u_{2}^{-3}$,
$u_{2}^{-2}$, $u_{2}^{-1}$ are all vanishing which implies that $U\left(
z;\tau \right)  $ is holomorphic at $\frac{\tau}{2}$.\smallskip

Now we prove $U\left(  z;\tau \right)  $ can be written as (\ref{384}). To
compute the coefficients $L\left(  \tau \right)  $, $M\left(  \tau \right)  $,
$N\left(  \tau \right)  $ and $C\left(  \tau \right)  $, we only need to compute
near $p$. By (\ref{390}) and (\ref{391}), near $p$, we have%
\begin{equation}
\Omega_{12}^{\prime}\left(  z;\tau \right)  =\frac{-i}{4\pi}\left(
\begin{array}
[c]{l}%
-u^{-2}-\left(  \wp \left(  2p\right)  +2\eta_{1}\right)  -\wp^{\prime}\left(
2p\right)  u\\
-\frac{1}{2}\left(  \frac{g_{2}}{10}+\wp^{\prime \prime}\left(  2p\right)
\right)  u^{2}+O\left(  u^{3}\right)
\end{array}
\right)  , \label{394}%
\end{equation}%
\begin{equation}
\Omega_{12}^{\prime \prime \prime}\left(  z;\tau \right)  =\frac{-i}{4\pi}\left(
-6u^{-4}-\left(  \frac{g_{2}}{10}+\wp^{\prime \prime}\left(  2p\right)
\right)  +O\left(  u\right)  \right)  , \label{395}%
\end{equation}%
\begin{equation}
I^{\prime}\left(  z;\tau \right)  =-\frac{3}{2}u^{-3}+Au^{-2}+H_{1}\left(
\tau \right)  +2H_{2}\left(  \tau \right)  u+O\left(  u^{2}\right)  ,
\label{396}%
\end{equation}%
\begin{align}
\frac{\partial}{\partial \tau}I\left(  z;\tau \right)   &  =\frac{3}{2}\frac
{dp}{d\tau}u^{-3}-A\frac{dp}{d\tau}u^{-2}-\frac{dA}{d\tau}u^{-1}\label{397}\\
&  +\left(  2A\frac{dA}{d\tau}-H_{1}\left(  \tau \right)  \frac{dp}{d\tau
}\right)  +O\left(  u\right)  .\nonumber
\end{align}
By (\ref{390}), (\ref{391}) and (\ref{394})-(\ref{397}), near $p$, after
computation, we have {\allowdisplaybreaks%
\begin{align}
&  U\left(  z;\tau \right) \label{403}\\
=  &  \left(  3\frac{dp}{d\tau}+\frac{i}{4\pi}\left[  6A-3\left(  \zeta \left(
2p\right)  -2p\eta_{1}\right)  \right]  \right)  u^{-3}\nonumber \\
&  +\left(  -2A\frac{dp}{d\tau}+\frac{i}{4\pi}\left[  -4A^{2}+2A\left(
\zeta \left(  2p\right)  -2p\eta_{1}\right)  \right]  \right)  u^{-2}%
\nonumber \\
&  +\left(  -2\frac{dA}{d\tau}+\frac{i}{4\pi}\left[
\begin{array}
[c]{l}%
4A\left(  \wp \left(  2p\right)  +\eta_{1}\right)  -3\wp^{\prime}\left(
2p\right) \\
-2\sum_{k=0}^{3}n_{k}\left(  n_{k}+1\right)  \wp^{\prime}\left(
p+\frac{\omega_{k}}{2}\right)
\end{array}
\right]  \right)  u^{-1}\nonumber \\
&  +\left(
\begin{array}
[c]{l}%
4A\frac{dA}{d\tau}-2H_{1}\left(  \tau \right)  \frac{dp}{d\tau}\\
+\frac{i}{4\pi}\left[
\begin{array}
[c]{l}%
-4A^{2}\left(  \wp \left(  2p\right)  +2\eta_{1}\right)  +3A\wp^{\prime}\left(
2p\right) \\
+2H_{1}\left(  \tau \right)  \left(  \zeta \left(  2p\right)  -2p\eta
_{1}\right)
\end{array}
\right]
\end{array}
\right)  +O\left(  u\right)  .\nonumber
\end{align}
}Obviously, (\ref{403}) implies that $U\left(  z;\tau \right)  $ has pole at
$p$ with order at most $3$. Since $U\left(  z;\tau \right)  $ is an even
elliptic function, $U\left(  z;\tau \right)  $ also has pole at $-p$ with order
at most $3$. From here and (\ref{408}), we conclude that $U\left(
z;\tau \right)  $ has poles only at $\pm p$ with order at most $3$. Moreover,
from (\ref{403}), it is easy to see that the coefficients $L\left(
\tau \right)  $, $M\left(  \tau \right)  $, $N\left(  \tau \right)  $ and
$C\left(  \tau \right)  $ are given by (\ref{398})-(\ref{402}).
\end{proof}

\begin{proof}
[Proof of Theorem \ref{thm4-2}]By Theorem \ref{thm M1} and Lemma \ref{lem-U},
the generalized Lam\'{e} equation (\ref{152}) with $\left(  p,A\right)
=\left(  p\left(  \tau \right)  ,A\left(  \tau \right)  \right)  $ is monodromy
preserving as $\tau$ deforms if and only if
\[
U\left(  z;\tau \right)  =0,
\]
if and only if
\begin{equation}
L\left(  \tau \right)  =M\left(  \tau \right)  =N\left(  \tau \right)  =C\left(
\tau \right)  =0. \label{403-1}%
\end{equation}
By (\ref{398})-(\ref{402}), a straightforward computation shows that
(\ref{403-1}) is equivalent to that $\left(  p,A\right)  =\left(  p\left(
\tau \right)  ,A\left(  \tau \right)  \right)  $ satisfies the Hamiltonian
system (\ref{142}) (see (\ref{144}) below).
\end{proof}

\subsection{Hamiltonian system and Painlev\'{e} VI}

Next, we will study the Hamiltonian structure for the elliptic form
(\ref{124}) with $\alpha_{i}$ defined by (\ref{125}). Our second main theorem
is the following:

\begin{theorem}
\label{thm4-1}The elliptic form (\ref{124}) with $\alpha_{k}=\frac{1}%
{2}\left(  n_{k}+\frac{1}{2}\right)  ^{2}$, $k=0,1,2,3$ is equivalent to the
Hamiltonian system defined by (\ref{142}) and (\ref{143}).
\end{theorem}

\begin{proof}
Suppose $\left(  p\left(  \tau \right)  ,A\left(  \tau \right)  \right)  $
satisfies the Hamiltonian system (\ref{142}), i.e.,%
\begin{equation}
\left \{
\begin{array}
[c]{l}%
\frac{dp\left(  \tau \right)  }{d\tau}=\frac{\partial K\left(  p,A,\tau \right)
}{\partial A}=\frac{-i}{4\pi}\left(  2A-\zeta \left(  2p|\tau \right)
+2p\eta_{1}\left(  \tau \right)  \right) \\
\frac{dA\left(  \tau \right)  }{d\tau}=-\frac{\partial K\left(  p,A,\tau
\right)  }{\partial p}=\frac{i}{4\pi}\left(
\begin{array}
[c]{l}%
\left(  2\wp \left(  2p|\tau \right)  +2\eta_{1}\left(  \tau \right)  \right)
A-\frac{3}{2}\wp^{\prime}\left(  2p|\tau \right) \\
-\sum_{k=0}^{3}n_{k}\left(  n_{k}+1\right)  \wp^{\prime}\left(  p+\frac
{\omega_{k}}{2}|\tau \right)
\end{array}
\right)
\end{array}
\right.  . \label{144}%
\end{equation}
Then we compute the second derivative $\frac{d^{2}p\left(  \tau \right)
}{d\tau^{2}}$ of $p\left(  \tau \right)  $ as follows:{\allowdisplaybreaks%
\begin{align}
\frac{d^{2}p\left(  \tau \right)  }{d\tau^{2}}=  &  \frac{-i}{4\pi}\left[
2\frac{dA\left(  \tau \right)  }{d\tau}+2\wp \left(  2p|\tau \right)
\frac{dp\left(  \tau \right)  }{d\tau}-\frac{\partial}{\partial \tau}%
\zeta \left(  2p|\tau \right)  \right. \label{144-1}\\
&  \left.  +2\eta_{1}\left(  \tau \right)  \frac{dp\left(  \tau \right)  }%
{d\tau}+2p\left(  \tau \right)  \frac{d\eta_{1}\left(  \tau \right)  }{d\tau
}\right]  .\nonumber
\end{align}
}By Lemma \ref{lem4-1}, we have%
\begin{equation}
-\frac{\partial}{\partial \tau}\zeta \left(  2p|\tau \right)  =\frac{-i}{4\pi
}\left[
\begin{array}
[c]{l}%
\wp^{\prime}\left(  2p|\tau \right)  +2\left(  \zeta \left(  2p|\tau \right)
-2p\eta_{1}\left(  \tau \right)  \right)  \wp \left(  2p|\tau \right) \\
+2\eta_{1}\zeta \left(  2p|\tau \right)  -\frac{1}{3}pg_{2}\left(  \tau \right)
\end{array}
\right]  , \label{148}%
\end{equation}%
\begin{equation}
2p\left(  \tau \right)  \frac{d\eta_{1}\left(  \tau \right)  }{d\tau}=\frac
{i}{2\pi}p\left(  \tau \right)  \left[  2\eta_{1}^{2}-\frac{1}{6}g_{2}\left(
\tau \right)  \right]  . \label{150}%
\end{equation}
Substituting (\ref{144}), (\ref{148}) and (\ref{150}) into (\ref{144-1}), we
have{\allowdisplaybreaks%
\begin{align}
\frac{d^{2}p\left(  \tau \right)  }{d\tau^{2}}  &  =\frac{-1}{4\pi^{2}}\left(
\wp^{\prime}\left(  2p|\tau \right)  +\frac{1}{2}\sum_{k=0}^{3}n_{k}\left(
n_{k}+1\right)  \wp^{\prime}\left(  p+\frac{\omega_{k}}{2}|\tau \right)
\right) \nonumber \\
&  =\frac{-1}{4\pi^{2}}\left[  \frac{1}{8}\sum_{k=0}^{3}\wp^{\prime}\left(
p+\frac{\omega_{k}}{2}|\tau \right)  +\frac{1}{2}\sum_{k=0}^{3}n_{k}\left(
n_{k}+1\right)  \wp^{\prime}\left(  p+\frac{\omega_{k}}{2}|\tau \right)
\right] \label{151}\\
&  =\frac{-1}{4\pi^{2}}\sum_{k=0}^{3}\frac{1}{2}\left(  n_{k}+\frac{1}%
{2}\right)  ^{2}\wp^{\prime}\left(  p+\frac{\omega_{k}}{2}|\tau \right)
,\nonumber
\end{align}
}implying that $p\left(  \tau \right)  $ is a solution of the elliptic form
(\ref{124}) with $\alpha_{k}=\frac{1}{2}\left(  n_{k}+\frac{1}{2}\right)
^{2}$, $k=0,1,2,3$.

Conversely, suppose $p\left(  \tau \right)  $ is a solution of the elliptic
form (\ref{124}) with $\alpha_{k}=\frac{1}{2}\left(  n_{k}+\frac{1}{2}\right)
^{2}$, $k=0,1,2,3$. We define $A\left(  \tau \right)  $ by the first equation
of (\ref{144}), i.e.,%
\begin{equation}
A\left(  \tau \right)  \doteqdot2\pi i\frac{dp\left(  \tau \right)  }{d\tau
}+\frac{1}{2}\left(  \zeta \left(  2p|\tau \right)  -2p\eta_{1}\left(
\tau \right)  \right)  . \label{145}%
\end{equation}
Then{\allowdisplaybreaks%
\begin{align*}
\frac{dA\left(  \tau \right)  }{d\tau}=  &  2\pi i\frac{d^{2}p\left(
\tau \right)  }{d\tau^{2}}+\frac{1}{2}\left(  -2\wp \left(  2p|\tau \right)
\frac{dp\left(  \tau \right)  }{d\tau}+\frac{\partial}{\partial \tau}%
\zeta \left(  2p|\tau \right)  \right) \\
&  -\left(  \eta_{1}\left(  \tau \right)  \frac{dp\left(  \tau \right)  }{d\tau
}+p\left(  \tau \right)  \frac{d\eta_{1}\left(  \tau \right)  }{d\tau}\right)  ,
\end{align*}
}and by using (\ref{144}), (\ref{148}) and (\ref{150}), we
have{\allowdisplaybreaks%
\begin{align*}
\frac{dA\left(  \tau \right)  }{d\tau}=  &  \frac{i}{4\pi}\left[  2\left(
\wp \left(  2p|\tau \right)  +\eta_{1}\left(  \tau \right)  \right)  A-\frac
{3}{2}\wp^{\prime}\left(  2p|\tau \right)  \right. \\
&  \left.  -\sum_{k=0}^{3}n_{k}\left(  n_{k}+1\right)  \wp^{\prime}\left(
p+\frac{\omega_{k}}{2}|\tau \right)  \right]  .
\end{align*}
}Thus, $\left(  p\left(  \tau \right)  ,A\left(  \tau \right)  \right)  $ is a
solution to the Hamiltonian system (\ref{144}).
\end{proof}

Moreover, from (\ref{144}), we could obtain the integral formula for $A\left(
\tau \right)  $ in Theorem \ref{theorem1-5}.

\begin{proof}
[Proof of Theorem \ref{theorem1-5}]Let us consider $F\left(  \tau \right)
=A+\frac{1}{2}\left(  \zeta \left(  2p\right)  -2\zeta \left(  p\right)
\right)  $ and compute $\frac{d}{d\tau}F\left(  \tau \right)  $. By (\ref{144})
and Lemma \ref{lem4-1}, we have{\allowdisplaybreaks%
\begin{align*}
&  \frac{d}{d\tau}F\left(  \tau \right) \\
=  &  \frac{dA}{d\tau}+\frac{1}{2}\frac{d}{d\tau}\left(  \zeta \left(
2p\right)  -2\zeta \left(  p\right)  \right) \\
=  &  \frac{i}{4\pi}\left(  2\left(  \wp \left(  2p\right)  +\eta_{1}\left(
\tau \right)  \right)  A-\frac{3}{2}\wp^{\prime}\left(  2p\right)  -\sum
_{k=0}^{3}n_{k}\left(  n_{k}+1\right)  \wp^{\prime}(p(\tau)+\frac{\omega_{k}%
}{2})\right) \\
&  -\left(  \wp \left(  2p\right)  -\wp \left(  p\right)  \right)  \frac
{dp}{d\tau}+\frac{1}{2}\left(  \frac{\partial}{\partial \tau}\zeta \left(
2p\right)  -2\frac{\partial}{\partial \tau}\zeta \left(  p\right)  \right) \\
=  &  \frac{i}{2\pi}\left(  2\wp \left(  2p\right)  -\wp \left(  p\right)
+\eta_{1}\right)  F\left(  \tau \right)  -\frac{i}{4\pi}\sum_{k=0}^{3}%
n_{k}\left(  n_{k}+1\right)  \wp^{\prime}(p(\tau)+\frac{\omega_{k}}{2}).
\end{align*}
}Thus,
\begin{align}
&  F\left(  \tau \right) \label{381}\\
=  &  \exp \left \{  \frac{i}{2\pi}\int^{\tau}\left(  2\wp \left(  2p(\hat{\tau
})|\hat{\tau}\right)  -\wp \left(  p(\hat{\tau})|\hat{\tau}\right)  +\eta
_{1}(\hat{\tau})\right)  d\hat{\tau}\right \}  \cdot J\left(  \tau \right)
,\nonumber
\end{align}
where
\[
J\left(  \tau \right)  =\int^{\tau}\frac{-\frac{i}{4\pi}\left(  \sum_{k=0}%
^{3}n_{k}(n_{k}+1)\wp^{\prime}(p(\hat{\tau})+\frac{\omega_{k}}{2}|\hat{\tau
})\right)  }{\exp \left \{  \frac{i}{2\pi}\int^{\hat{\tau}}\left(  2\wp
(2p(\tau^{\prime})|\tau^{\prime})-\wp(p(\tau^{\prime})|\tau^{\prime})+\eta
_{1}\left(  \tau^{\prime}\right)  \right)  d\tau^{\prime}\right \}  }d\hat
{\tau}+c_{1}%
\]
for some constant $c_{1}\in \mathbb{C}$. By Lemma \ref{lem4-1}, we have
\begin{equation}
\frac{3i}{4\pi}\int^{\tau}\eta_{1}(\hat{\tau})d\hat{\tau}=\ln \theta
_{1}^{\prime}(\tau). \label{theta}%
\end{equation}
Then (\ref{514}) follows from (\ref{381}) and (\ref{theta}).
\end{proof}

\section{Collapse of two singular points}

In this section, we study the phenomena of collapsing two singular points $\pm
p(\tau)$ to $0$ in the generalized Lam\'{e} equation (\ref{89-0}) when
$p(\tau)$ is a solution of the elliptic form (\ref{124}). As an application of
Theorems \ref{thm4-2} and \ref{thm4-1}, it turns out that the classical
Lam\'{e} equation%
\begin{equation}
y^{\prime \prime}(z)=\left(  n(n+1)\wp(z)+B\right)  y(z)\text{ \ in \ }E_{\tau}
\label{5033}%
\end{equation}
appears as a limiting equation if $n_{k}=0$ for $k=1,2,3$ (see Theorem
\ref{thm-II-9} below). First we recall the following classical result.\medskip \ 

\noindent \textbf{Theorem A.\cite[Proposition 1.4.1]{GP} }\textit{Assume
}$\theta_{4}=n_{0}+\frac{1}{2}\not =0$. \textit{Then for any }$t_{0}%
\in \mathbb{CP}^{1}\backslash \{0,1,\infty \}$\textit{, there exist two }%
$1$-\textit{parameter families of solutions }$\lambda(t)$\textit{ of
Painlev\'{e} VI (\ref{46}) such that}%
\begin{equation}
\lambda(t)=\frac{\beta}{t-t_{0}}+h+O(t-t_{0})\text{ as }t\rightarrow t_{0},
\label{II-132}%
\end{equation}
\textit{where }$h\in \mathbb{C}$ \textit{can be taken arbitrary and}%
\begin{equation}
\beta=\beta(\theta,t_{0})\in \left \{  \pm \tfrac{t_{0}(t_{0}-1)}{\theta_{4}%
}\right \}  . \label{II-133}%
\end{equation}
\textit{Furthermore, these two }$1$-\textit{parameter families of solutions
give all solutions of Painlev\'{e} VI (\ref{46}) which has a pole at }$t_{0}%
$\textit{.\medskip}

In this paper, we always identify the solutions $p(\tau)$ and $-p(\tau)$ of
the elliptic form (\ref{124}). As a consequence of Theorem A and the
transformation (\ref{tr}), we have the following result.

\begin{lemma}
\label{lemII-4} \textit{Assume }$n_{0}+\frac{1}{2}\not =0$. \textit{Then for
any }$\tau_{0}\in \mathbb{H}$\textit{, by the transformation (\ref{tr})
solutions }$\lambda(t)$\textit{ in Theorem A give two }$1$-\textit{parameter
families of solutions }$p(\tau)$\textit{ of the elliptic form (\ref{124}) such
that}%
\begin{equation}
p(\tau)=c_{0}(\tau-\tau_{0})^{\frac{1}{2}}(1+\tilde{h}(\tau-\tau_{0}%
)+O(\tau-\tau_{0})^{2})\text{ as }\tau \rightarrow \tau_{0}, \label{515-5}%
\end{equation}
where $\tilde{h}\in \mathbb{C}$ \textit{can be taken arbitrary,}%
\begin{equation}
c_{0}^{2}=\left \{
\begin{array}
[c]{l}%
i\frac{n_{0}+\frac{1}{2}}{\pi}\text{ \  \ if \  \ }\beta=-\frac{t_{0}(t_{0}%
-1)}{\theta_{4}}\\
-i\frac{n_{0}+\frac{1}{2}}{\pi}\text{ \  \ if \  \ }\beta=\frac{t_{0}(t_{0}%
-1)}{\theta_{4}}%
\end{array}
\right.  , \label{515-3}%
\end{equation}
and $t_{0}=t(\tau_{0})$. \textit{Furthermore, these two }$1$-\textit{parameter
families of solutions give all solutions }$p(\tau)$\textit{ of the elliptic
form (\ref{124}) such that }$p(\tau_{0})=0$\textit{.}
\end{lemma}

\begin{proof}
It suffices to prove (\ref{515-3}), which follows readily from
\begin{equation}
t^{\prime}(\tau_{0})=-i\frac{t_{0}(t_{0}-1)}{\pi}\left(  e_{2}(\tau_{0}%
)-e_{1}(\tau_{0})\right)  . \label{II-134}%
\end{equation}
Remark that $t_{0}\not \in \{0,1\}$, so (\ref{II-134}) implies $t^{\prime
}(\tau_{0})\not =0$.

The formula (\ref{II-134}) is known in the literature. Here we give a proof
for the reader's convenience. Recalling theta functions $\vartheta_{2}%
(\tau),\vartheta_{3}(\tau)$ and $\vartheta_{4}(\tau)$, it is well-known that
(cf. see \cite{YB} for a reference)%
\begin{equation}
e_{3}(\tau)-e_{2}(\tau)=\pi^{2}\vartheta_{2}(\tau)^{4},\text{ \ }e_{1}%
(\tau)-e_{3}(\tau)=\pi^{2}\vartheta_{4}(\tau)^{4}, \label{II-134-2}%
\end{equation}%
\[
e_{1}(\tau)-e_{2}(\tau)=\pi^{2}\vartheta_{3}(\tau)^{4},
\]%
\[
\frac{d}{d\tau}\ln \vartheta_{4}(\tau)=\frac{i}{12\pi}\left[  3\eta_{1}%
(\tau)-\pi^{2}(2\vartheta_{2}(\tau)^{4}+\vartheta_{4}(\tau)^{4})\right]  ,
\]%
\[
\frac{d}{d\tau}\ln \vartheta_{3}(\tau)=\frac{i}{12\pi}\left[  3\eta_{1}%
(\tau)+\pi^{2}(\vartheta_{2}(\tau)^{4}-\vartheta_{4}(\tau)^{4})\right]  .
\]
Therefore, $t=\vartheta_{4}^{4}/\vartheta_{3}^{4}$ and then%
\begin{align}
t^{\prime}(\tau)  &  =4t\left(  \frac{d}{d\tau}\ln \vartheta_{4}-\frac{d}%
{d\tau}\ln \vartheta_{3}\right)  =-i\pi t\cdot \vartheta_{2}^{4}\label{II-134-1}%
\\
&  =-i\pi \frac{\vartheta_{2}^{4}\vartheta_{4}^{4}}{\vartheta_{3}^{4}}%
=-i\frac{t(t-1)}{\pi}\left(  e_{2}-e_{1}\right)  .\nonumber
\end{align}
This completes the proof.
\end{proof}

\begin{theorem}
\label{thm-II-9}Assume that $n_{k}\not \in \mathbb{Z}+\frac{1}{2},$
$k\in \{0,1,2,3\}$, and (\ref{101}) hold. Let $(p(\tau),A(\tau))$ be a solution
of the Hamiltonian system (\ref{aa}) such that $p(\tau_{0})=0$ for some
$\tau_{0}\in \mathbb{H}$. Then%
\begin{equation}
p(\tau)=c_{0}(\tau-\tau_{0})^{\frac{1}{2}}(1+\tilde{h}(\tau-\tau_{0}%
)+O(\tau-\tau_{0})^{2})\text{ as }\tau \rightarrow \tau_{0}, \label{515-4}%
\end{equation}
for some $\tilde{h}\in \mathbb{C}$ and the generalized Lam\'{e} equation
(\ref{89-0}) converges to%
\begin{equation}
y^{\prime \prime}=\left[  \sum_{j=1}^{3}n_{j}\left(  n_{j}+1\right)  \wp \left(
z+\frac{\omega_{j}}{2}\right)  +m(m+1)\wp(z)+B_{0}\right]  y\text{ in }%
E_{\tau_{0}}, \label{503-2}%
\end{equation}
where $c_{0}$ is seen in (\ref{515-3}),%
\begin{equation}
m=\left \{
\begin{array}
[c]{l}%
n_{0}+1\text{ \ if \ }c_{0}^{2}=i\frac{n_{0}+\frac{1}{2}}{\pi}\text{\ i.e.,\ }%
\beta=-\frac{t_{0}(t_{0}-1)}{n_{0}+\frac{1}{2}},\\
n_{0}-1\text{ \ if \ }c_{0}^{2}=-i\frac{n_{0}+\frac{1}{2}}{\pi}%
\text{\ i.e.,\ }\beta=\frac{t_{0}(t_{0}-1)}{n_{0}+\frac{1}{2}},
\end{array}
\right.  \label{503-3}%
\end{equation}%
\begin{equation}
B_{0}=2\pi ic_{0}^{2}\left(  4\pi i\tilde{h}-\eta_{1}(\tau_{0})\right)
-\sum_{j=1}^{3}n_{j}(n_{j}+1)e_{j}(\tau_{0}). \label{503-4}%
\end{equation}

\end{theorem}

\begin{proof}
Clearly (\ref{515-4}) follows from Lemma \ref{lemII-4}, by which we have
(write $p=p(\tau)$)%
\[
(\tau-\tau_{0})^{\frac{1}{2}}=\frac{1}{c_{0}}p\left(  1-\frac{1}{c_{0}^{2}%
}\tilde{h}p^{2}+O(p^{4})\right)  \text{ as }\tau \rightarrow \tau_{0}.
\]
Consequently,{\allowdisplaybreaks}%
\begin{align*}
p^{\prime}(\tau)  &  =\frac{1}{2}c_{0}(\tau-\tau_{0})^{-\frac{1}{2}}%
[1+3\tilde{h}(\tau-\tau_{0})+O((\tau-\tau_{0})^{2})]\\
&  =\frac{c_{0}^{2}}{2p}\left[  1+\frac{4}{c_{0}^{2}}\tilde{h}p^{2}%
+O(p^{4})\right]  \text{\ as\ }\tau \rightarrow \tau_{0}.
\end{align*}
This, together with the first equation of the Hamiltonian system
(\ref{142-0})-(\ref{143-0}), gives{\allowdisplaybreaks%
\begin{align}
A(\tau)  &  =\frac{1}{2}\left[  4\pi ip^{\prime}(\tau)+\zeta(2p(\tau
))-2p(\tau)\eta_{1}(\tau)\right] \label{II-123}\\
&  =\frac{\pi ic_{0}^{2}}{p}\left[  1+\frac{4}{c_{0}^{2}}\tilde{h}%
p^{2}+O(p^{4})\right]  +\frac{1}{4p}-\eta_{1}(\tau_{0})p+O(p^{3})\nonumber \\
&  =\frac{c}{p}+ep+O(p^{3})\text{ \ as \ }\tau \rightarrow \tau_{0},\nonumber
\end{align}
}where $e=4\pi i\tilde{h}-\eta_{1}(\tau_{0})$ and%
\begin{equation}
c=\pi ic_{0}^{2}+\frac{1}{4}=\left \{
\begin{array}
[c]{c}%
-n_{0}-\frac{1}{4}\text{ \  \ if \  \ }c_{0}^{2}=i\frac{n_{0}+\frac{1}{2}}{\pi
},\\
n_{0}+\frac{3}{4}\text{ \  \ if \  \ }c_{0}^{2}=-i\frac{n_{0}+\frac{1}{2}}{\pi}.
\end{array}
\right.  \label{II-124}%
\end{equation}
Clearly $c$ satisfies
\[
c^{2}-\frac{c}{2}-\frac{3}{16}-n_{0}(n_{0}+1)=0.
\]
Consequently, we have{\allowdisplaybreaks%
\begin{align*}
B(\tau)=  &  A^{2}-\zeta \left(  2p\right)  A-\frac{3}{4}\wp \left(  2p\right)
-\sum_{j=0}^{3}n_{j}\left(  n_{j}+1\right)  \wp \left(  p+\frac{\omega_{j}}%
{2}\right) \\
=  &  \left(  \frac{c}{p}+ep+O(p^{3})\right)  ^{2}-\left(  \frac{1}%
{2p}+O(p^{3})\right)  \left(  \frac{c}{p}+ep+O(p^{3})\right) \\
&  -\frac{\left(  \frac{3}{16}+n_{0}(n_{0}+1)\right)  }{p^{2}}-\sum_{j=1}%
^{3}n_{j}\left(  n_{j}+1\right)  e_{j}(\tau_{0})+O(p^{2})\\
=  &  \frac{4c-1}{2}e-\sum_{j=1}^{3}n_{j}\left(  n_{j}+1\right)  e_{j}%
(\tau_{0})+O(p^{2})=B_{0}+O(p^{2})
\end{align*}
}as $p=p(\tau)\rightarrow0$ since $\tau \rightarrow \tau_{0}$, where $B_{0}$ is
given by (\ref{503-4}). Furthermore, (\ref{II-123}) implies%
\[
A(\zeta(z+p)-\zeta(z-p))=A(-2p\wp(z)+O(p^{2}))\rightarrow-2c\wp(z)
\]
uniformly for $z$ bounded away from the lattice points. Therefore, the
potential of the generalized Lam\'{e} equation (\ref{89-0}) converges
to{\allowdisplaybreaks%
\[
\sum_{j=1}^{3}n_{j}\left(  n_{j}+1\right)  \wp \left(  z+\frac{\omega_{j}}%
{2}\right)  +\left[  n_{0}(n_{0}+1)+\frac{3}{2}-2c\right]  \wp(z)+B_{0}%
\]
}uniformly for $z$ bounded away from the lattice points as $\tau
\rightarrow \tau_{0}$. Using (\ref{II-124}) we easily obtain (\ref{503-2}%
)-(\ref{503-3}).
\end{proof}

\section{Correspondence between generalized Lam\'{e} equation and Fuchsian
equation}

In this section, we want to establish a one to one correspondence between the
generalized Lam\'{e} equation (\ref{89-0}) and a type of Fuchsian equations on
$\mathbb{CP}^{1}$. After the correspondence, naturally we ask the question:
\textit{Is the isomonodromic deformation for the generalized Lam\'{e} equation
in }$E_{\tau}$\textit{ equivalent to the isomonodromic deformation for the
corresponding Fuchsian equation on }$\mathbb{CP}^{1}$\textit{?} Notice that we
establish the correspondence by using the transformation $x=\frac{\wp \left(
z\right)  -e_{1}}{e_{2}-e_{1}}$ (see (\ref{123}) below) which is a double
cover from $E_{\tau}$ onto $\mathbb{CP}^{1}$. Hence, it is clear that the
isomonodromic deformation for the Fuchsian equation could imply the the
isomonodromic deformation for the generalized Lam\'{e} equation. However, the
converse assertion is not easy at all, because the lifting of a closed loop in
$\mathbb{CP}^{1}$ via $x=\frac{\wp \left(  z\right)  -e_{1}}{e_{2}-e_{1}}$ is
not necessarily a closed loop in $E_{\tau}$. As an application of Theorems
\ref{thm4-2} and \ref{thm4-1}, we could give a positive answer.

First we review the Fuchs-Okamoto theory. Consider a second order Fuchsian
equation defined on $\mathbb{CP}^{1}$ as follows:%
\begin{equation}
y^{\prime \prime}+p_{1}\left(  x\right)  y^{\prime}+p_{2}\left(  x\right)  y=0,
\label{90}%
\end{equation}
which has five regular singular points at $\left \{  t,0,1,\infty
,\lambda \right \}  $ and $p_{j}\left(  x\right)  =p_{j}(x;t$, $\lambda$, $\mu
)$, $j=1,2$, are rational functions in $x$ such that the Riemann scheme of
(\ref{90}) is%
\begin{equation}
\left(
\begin{array}
[c]{ccccc}%
t & 0 & 1 & \infty & \lambda \\
0 & 0 & 0 & \hat{\alpha} & 0\\
\theta_{t} & \theta_{0} & \theta_{1} & \hat{\alpha}+\theta_{\infty} & 2
\end{array}
\right)  , \label{91}%
\end{equation}
where $\hat{\alpha}$ is determined by the Fuchsian relation, that is,
\[
\hat{\alpha}=-\tfrac{1}{2}\left(  \theta_{t}+\theta_{0}+\theta_{1}%
+\theta_{\infty}-1\right)  .
\]
Throughout this section we always assume that
\begin{equation}
\text{ }\lambda \not \in \left \{  0,1,t\right \}  \text{ and }\lambda \text{
\textit{is an apparent singular point}.} \label{137}%
\end{equation}
Since one exponent at any one of $0,1,\lambda,t$ is $0$ (see (\ref{91})),
$p_{2}(x)$ has only simple poles at $0,1,\lambda,t$. The residue of
$p_{1}\left(  x\right)  $ at $x=\lambda$ is $-1$ because another exponent at
$x=\lambda$ is $2$. Define $\mu$ and $K$ as follows:%
\begin{equation}
\mu \doteqdot \underset{x=\lambda}{\text{ Res }}p_{2}\left(  x\right)  ,\text{
}K\doteqdot-\underset{x=t}{\text{Res}}\text{ }p_{2}\left(  x\right)  .
\label{92}%
\end{equation}
By (\ref{91})-(\ref{92}), we have
\begin{equation}
p_{1}\left(  x\right)  =\frac{1-\theta_{t}}{x-t}+\frac{1-\theta_{0}}{x}%
+\frac{1-\theta_{1}}{x-1}-\frac{1}{x-\lambda}, \label{96}%
\end{equation}%
\begin{equation}
p_{2}\left(  x\right)  =\frac{\hat{\kappa}}{x\left(  x-1\right)  }%
-\frac{t\left(  t-1\right)  K}{x\left(  x-1\right)  \left(  x-t\right)
}+\frac{\lambda \left(  \lambda-1\right)  \mu}{x\left(  x-1\right)  \left(
x-\lambda \right)  }, \label{97}%
\end{equation}
\medskip where
\begin{equation}
\hat{\kappa}=\hat{\alpha}\left(  \hat{\alpha}+\theta_{\infty}\right)
=\frac{1}{4}\left \{  \left(  \theta_{0}+\theta_{1}+\theta_{t}-1\right)
^{2}-\theta_{\infty}^{2}\right \}  . \label{a}%
\end{equation}
By the condition (\ref{137}), i.e., $\lambda$ is apparent, $K$ can be
expressed explicitly by%

\begin{equation}
K\left(  \lambda,\mu,t\right)  =\frac{1}{t\left(  t-1\right)  }\left \{
\begin{array}
[c]{l}%
\lambda \left(  \lambda-1\right)  \left(  \lambda-t\right)  \mu^{2}+\hat
{\kappa}\left(  \lambda-t\right) \\
-\left[
\begin{array}
[c]{l}%
\theta_{0}\left(  \lambda-1\right)  \left(  \lambda-t\right)  +\theta
_{1}\lambda \left(  \lambda-t\right) \\
+\left(  \theta_{t}-1\right)  \lambda \left(  \lambda-1\right)
\end{array}
\right]  \mu
\end{array}
\right \}  . \label{98}%
\end{equation}
For all details about (\ref{96})-(\ref{98}), we refer the reader to \cite{GP}.

Now let $t$ be the deformation parameter, and assume that (\ref{90}) with
$\left(  \lambda \left(  t\right)  ,\mu \left(  t\right)  \right)  $ preserves
the monodromy representation. In \cite{Fuchs, Okamoto2}, it was discovered
that under the non-resonant condition, $\left(  \lambda \left(  t\right)
,\mu \left(  t\right)  \right)  $ must satisfy the following Hamiltonian
system:%
\begin{equation}
\frac{d\lambda \left(  t\right)  }{dt}=\frac{\partial K}{\partial \mu},\text{
\ }\frac{d\mu \left(  t\right)  }{dt}=-\frac{\partial K}{\partial \lambda}.
\label{aa}%
\end{equation}
Indeed, the following theorem was proved in \cite{Fuchs, Okamoto2}.\medskip

\noindent \textbf{Theorem B.\cite{Fuchs, Okamoto2} }\emph{Suppose that }%
$\theta_{t},\theta_{0},\theta_{1},\theta_{\infty}\notin \mathbb{Z}$ \emph{(i.e.
the non-resonant condition)} \emph{and }$\lambda$\emph{\ is an apparent
singular point. Then the second order ODE (\ref{90}) preserves the monodromy
as }$t$\emph{\ deforms if and only if }$\left(  \lambda \left(  t\right)
,\mu \left(  t\right)  \right)  $\emph{\ satisfies the Hamiltonian system
(\ref{aa}).\medskip}

It is well-known in the literature that a solution of Painlev\'{e} VI
(\ref{46}) can be obtained from the Hamiltonian system (\ref{aa}) with the
Hamiltonian $K\left(  \lambda,\mu,t\right)  $ defined in (\ref{98}). Let
$\left(  \lambda \left(  t\right)  ,\mu \left(  t\right)  \right)  $ be a
solution to the Hamiltonian system (\ref{aa}). Then $\lambda \left(  t\right)
$ satisfies the Painlev\'{e} VI (\ref{46}) with parameters%
\begin{equation}
\left(  \alpha,\beta,\gamma,\delta \right)  =\left(  \tfrac{1}{2}\theta
_{\infty}^{2},\, -\tfrac{1}{2}\theta_{0}^{2},\, \tfrac{1}{2}\theta_{1}^{2},\,
\tfrac{1}{2}\left(  1-\theta_{t}^{2}\right)  \right)  . \label{67}%
\end{equation}

Conversely, if $\lambda \left(  t\right)  $ is a solution to Painlev\'{e} VI
(\ref{46}), then we define $\mu \left(  t\right)  $ by the first equation of
(\ref{aa}), where $(\theta_{0},\theta_{1},\theta_{t},\theta_{\infty})$ and
$\hat{\kappa}$ are given by (\ref{67}) and (\ref{a}), respectively.
Consequently, $\left(  \lambda \left(  t\right)  ,\mu \left(  t\right)  \right)
$ is a solution to (\ref{aa}). The above facts can be proved directly. For
details, we refer the reader to \cite{GR, GP}. Together with this fact and
Theorem B, we have\medskip

\noindent \textbf{Theorem C.} \emph{Assume the same hypotheses of Theorem B.
Then the second order ODE (\ref{90}) preserves the monodromy as }$t$\emph{
deforms if and only if }$\lambda(t)$\emph{ satisfies Painlev\'{e} VI
(\ref{46}) with parameters (\ref{67}).\medskip}

Now let us consider the following generalized Lam\'{e} equation in $E_{\tau}$:%
\begin{equation}
y^{\prime \prime}=\left[
\begin{array}
[c]{l}%
\sum_{i=0}^{3}n_{i}\left(  n_{i}+1\right)  \wp \left(  z+\frac{\omega_{i}}%
{2}\right)  +\frac{3}{4}\left(  \wp \left(  z+p\right)  +\wp \left(  z-p\right)
\right) \\
+A\left(  \zeta \left(  z+p\right)  -\zeta \left(  z-p\right)  \right)  +B
\end{array}
\right]  y, \label{89}%
\end{equation}
and suppose that $p$ is an apparent singularity of (\ref{89}). Then we shall
prove that the generalized Lam\'{e} equation (\ref{89}) is 1-1 correspondence
to the 2nd order Fuchsian equation (\ref{90}) with $\lambda$ being an apparent
singularity. To describe the 1-1 correspondence between (\ref{90}) and
(\ref{89}), we set%
\begin{equation}
x=\frac{\wp \left(  z\right)  -e_{1}}{e_{2}-e_{1}}\text{ and\ }\mathfrak{p}%
\left(  x\right)  =4x\left(  x-1\right)  \left(  x-t\right)  . \label{123}%
\end{equation}
Then we have the following theorem:

\begin{theorem}
\label{thm1}Given a generalized Lam\'{e} equation (\ref{89}) defined in
$E_{\tau}$. Suppose $p\not \in E_{\tau}\left[  2\right]  $ is an apparent
singularity of (\ref{89}). Then by using $x=\frac{\wp \left(  z\right)  -e_{1}%
}{e_{2}-e_{1}}$, there is a corresponding 2nd order Fuchsian equation
(\ref{90}) satisfying (\ref{91}) and (\ref{137}) whose coefficients
$p_{1}\left(  x\right)  $ and $p_{2}\left(  x\right)  $ are expressed by
(\ref{96})-(\ref{98}), where%
\begin{equation}
t=\frac{e_{3}-e_{1}}{e_{2}-e_{1}},\text{ \ }\lambda=\frac{\wp \left(  p\right)
-e_{1}}{e_{2}-e_{1}}, \label{138}%
\end{equation}%
\begin{equation}
\left(  \theta_{0},\theta_{1},\theta_{t},\theta_{\infty}\right)  =\left(
n_{1}+\tfrac{1}{2},n_{2}+\tfrac{1}{2},n_{3}+\tfrac{1}{2},n_{0}+\tfrac{1}%
{2}\right)  , \label{103}%
\end{equation}%
\begin{equation}
\hat{\alpha}=-\frac{1}{2}\left(  1+n_{0}+n_{1}+n_{2}+n_{3}\right)  ,
\label{139}%
\end{equation}%
\begin{equation}
\mu=\frac{2n_{3}-1}{4\left(  \lambda-t\right)  }+\frac{2n_{2}-1}{4\left(
\lambda-1\right)  }+\frac{2n_{1}-1}{4\lambda}+\frac{3}{8}\frac{\mathfrak{p}%
^{\prime}\left(  \lambda \right)  }{\mathfrak{p}\left(  \lambda \right)  }%
+\frac{A\wp^{\prime}\left(  p\right)  }{b^{2}\mathfrak{p}\left(
\lambda \right)  }, \label{104}%
\end{equation}%
\begin{align}
K  &  =-\frac{2n_{2}n_{3}-n_{2}-n_{3}}{4\left(  t-1\right)  }-\frac
{2n_{1}n_{3}-n_{1}-n_{3}}{4t}-\frac{2n_{3}-1}{4\left(  t-\lambda \right)
}\label{105}\\
&  +\frac{1}{4t\left(  t-1\right)  }\left[
\begin{array}
[c]{l}%
\frac{3}{2}\frac{\lambda \left(  \lambda-1\right)  }{\left(  \lambda-t\right)
}-\frac{3}{2}\frac{\wp \left(  p\right)  +e_{3}}{e_{2}-e_{1}}+\frac
{A\wp^{\prime}\left(  p\right)  }{\left(  \lambda-t\right)  (e_{2}-e_{1})^{2}%
}\\
+\frac{n_{0}\left(  n_{0}+1\right)  e_{3}}{e_{2}-e_{1}}+\frac{n_{1}\left(
n_{1}+1\right)  e_{2}}{e_{2}-e_{1}}+\frac{n_{2}\left(  n_{2}+1\right)  e_{1}%
}{e_{2}-e_{1}}\\
-\frac{2n_{3}\left(  n_{3}+1\right)  e_{3}}{e_{2}-e_{1}}+\frac{2}{e_{2}-e_{1}%
}A\wp \left(  p\right)  +\frac{B}{e_{2}-e_{1}}%
\end{array}
\right]  .\nonumber
\end{align}
Conversely, given a 2nd order Fuchsian equation (\ref{90}) satisfying
(\ref{91}) and (\ref{137}), there is a corresponding generalized Lam\'{e}
equation (\ref{89}) defined in $E_{\tau}$ where $\tau,\pm p,n_{i}$ are defined
by (\ref{138})-(\ref{103}), the constant $A$ is defined by solving
(\ref{104}), and the constant $B$ is defined by (\ref{101}). In particular,
$p$ is an apparent singularity of (\ref{89}).
\end{theorem}

\begin{remark}
For the second part of Theorem \ref{thm1}, the condition $\lambda
\not \in \{0,1,t\}$ is equivalent to $p\not \in E_{\tau}[2]$, which implies
$\wp^{\prime}(p)\not =0$. Thus, $A$ is well-defined via (\ref{104}). The proof
of Theorem \ref{thm1} will be given after Corollary \ref{corcor}.
\end{remark}

Let $\rho:\pi_{1}(E_{\tau}\backslash \left(  E_{\tau}\left[  2\right]
\cup \left \{  \pm p\right \}  \right)  ,q_{0})\rightarrow SL\left(
2,\mathbb{C}\right)  $, $\tilde{\rho}:\pi_{1}(\mathbb{CP}^{1}\backslash
\left \{  0,1,t,\infty \right \}  $, $\lambda_{0})\rightarrow GL\left(
2,\mathbb{C}\right)  $ where $\lambda_{0}=x\left(  q_{0}\right)  ,$ be the
monodromy representations of the generalized Lam\'{e} equation (\ref{89}) and
the corresponding Fuchsian equation (\ref{90}) respectively. Let $Y\left(
z\right)  =\left(  y_{1}\left(  z\right)  ,y_{2}\left(  z\right)  \right)  $
be a fixed fundamental solution of (\ref{89}). We denote $\mathcal{N}$ and
$\mathcal{M}$ to be the monodromy groups of (\ref{89}) and (\ref{90}) with
respect to $Y\left(  z\right)  $ and $\hat{Y}\left(  x\right)  $ respectively.
Here $\hat{Y}\left(  x\right)  =\left(  \hat{y}_{1}\left(  x\right)  ,\hat
{y}_{2}\left(  x\right)  \right)  $ with $\hat{y}_{j}\left(  x\right)  $
defined by%
\begin{align*}
{y_{j}\left(  z\right)  } &  =\psi(x){}\hat{y}_{j}\left(  x\right)  \\
&  \doteqdot \left(  x-\lambda \right)  ^{-\frac{1}{2}}x^{-\frac{n_{1}}{2}%
}(x-1)^{-\frac{n_{2}}{2}}(x-t)^{-\frac{n_{3}}{2}}\hat{y}_{j}\left(  x\right)
\text{, }j=1,2,
\end{align*}
and $x=\frac{\wp \left(  z\right)  -e_{1}}{e_{2}-e_{1}}$ is a fundamental
solution of equation (\ref{90}); see the proof of Theorem \ref{thm1} below.
Let $\gamma_{1}\in \pi_{1}\left(  E_{\tau}\backslash \left(  E_{\tau}\left[
2\right]  \cup \left \{  \pm p\right \}  \right)  ,q_{0}\right)  $ be a loop
which encircles the singularity $\frac{\omega_{1}}{2}$ once. Then $x\left(
\gamma_{1}\right)  \in \pi_{1}\left(  \mathbb{CP}^{1}\backslash \left \{
0,1,t,\infty \right \}  ,\lambda_{0}\right)  $. Since $x=\frac{\wp \left(
z\right)  -e_{1}}{e_{2}-e_{1}}$ is a double cover, the loop $x\left(
\gamma_{1}\right)  $ encircles the singularity $0$ twice. Thus, $x\left(
\gamma_{1}\right)  =\beta^{2}$ for some $\beta \in \pi_{1}\left(  \mathbb{CP}%
^{1}\backslash \left \{  0,1,t,\infty \right \}  ,\lambda_{0}\right)  $. Let
$\rho \left(  \gamma_{1}\right)  =N_{1}$ and $\tilde{\rho}\left(  \beta \right)
=M_{0}$. Then%
\begin{align}
Y\left(  z\right)  N_{1} &  =\gamma_{1}^{\ast}Y\left(  z\right)  =\left(
\beta^{2}\right)  ^{\ast}\left(  {\psi(x)}\hat{Y}\left(  x\right)  \right)
\label{1}\\
&  =C\left(  \beta^{2}\right)  {\psi(x)}\hat{Y}\left(  x\right)  M_{0}%
^{2}\nonumber \\
&  =Y\left(  z\right)  C\left(  \beta^{2}\right)  M_{0}^{2}\nonumber
\end{align}
for some constant $C\left(  \beta^{2}\right)  \in \mathbb{C}$ which comes from
the analytic continuation of ${\psi(x)}$ along $\beta^{2}$. From (\ref{1}), we
see that $N_{1}=C\left(  \beta^{2}\right)  M_{0}^{2}$. By the same argument,
we know that any element $N\in \mathcal{N}$ could be written as
\begin{equation}
N=CM_{1}M_{2}\label{2}%
\end{equation}
for some $M_{i}\in \mathcal{M}$, $i=1,2$ and some constant $C\in \mathbb{C}$
coming from the gauge transformation ${\psi(x)}$. In general, $\mathcal{N}$ is
not a subgroup of $\mathcal{M}$ because of ${\psi(x)}$. By (\ref{2}), the
isomonodromic deformation of (\ref{90}) implies the isomonodromic deformation
of (\ref{89}). However, it is not clear to see whether the converse assertion
is true or not from (\ref{2}). Here we can give a confirmative answer. In
fact, by (\ref{125}), (\ref{126}) and (\ref{67}), we have (\ref{103}) holds.
Since $n_{i}\not \in \frac{1}{2}+\mathbb{Z}$ for $i\in \{0,1,2,3\}$, we have
$\theta_{0},\theta_{1},\theta_{t},\theta_{\infty}\notin \mathbb{Z}$, i.e.,
non-resonant. Then as a consequence of Theorems \ref{thm1}, \ref{theorem1-2}
and C, we have

\begin{corollary}
\label{corcor}Suppose $n_{i}\not \in \frac{1}{2}+\mathbb{Z}$ for $i=0,1,2,3$.
If the generalized Lam\'{e} equation (\ref{89}) in $E_{\tau}$ preverses the
monodromy, then so does the corresponding Fuchsian equation (\ref{90}) on
$\mathbb{CP}^{1},$ and vice versa.
\end{corollary}

\begin{proof}
[Proof of Theorem \ref{thm1}]{ Let us first consider the generalized Lam\'{e}
equation (\ref{89}). By applying
\[
x=\frac{\wp \left(  z\right)  -e_{1}}{e_{2}-e_{1}},\text{ \ }t=\frac
{e_{3}-e_{1}}{e_{2}-e_{1}},\text{ \ }\lambda=\frac{\wp \left(  p\right)
-e_{1}}{e_{2}-e_{1}},
\]
and the addition formula%
\begin{equation}
\wp \left(  z+p\right)  +\wp \left(  z-p\right)  =\frac{\wp^{\prime}\left(
z\right)  ^{2}+\wp^{\prime}\left(  p\right)  ^{2}}{2\left(  \wp \left(
z\right)  -\wp \left(  p\right)  \right)  ^{2}}-2\wp \left(  z\right)
-2\wp \left(  p\right)  , \label{362}%
\end{equation}
the equation (\ref{89}) becomes the following second order Fuchsian equation
defined on $\mathbb{CP}^{1}$:%
\begin{equation}
y^{\prime \prime}(x)+\frac{1}{2}\frac{\mathfrak{p}^{\prime}\left(  x\right)
}{\mathfrak{p}\left(  x\right)  }y^{\prime}(x)-\frac{q(x)}{\mathfrak{p}\left(
x\right)  }y(x)=0, \label{106}%
\end{equation}
where $\mathfrak{p}(x)$ is defined in (\ref{123}), $b\doteqdot e_{2}-e_{1}$
and}%
\[
{\small q(x)=\left[
\begin{array}
[c]{l}%
n_{0}\left(  n_{0}+1\right)  \left(  x+\frac{e_{1}}{b}\right)  +\frac
{n_{1}\left(  n_{1}+1\right)  }{2}\left(  \frac{\mathfrak{p}\left(  x\right)
}{2x^{2}}-2x+\frac{4e_{1}}{b}\right)  +\frac{B}{b}\\
+\frac{n_{2}\left(  n_{2}+1\right)  }{2}\left(  \frac{\mathfrak{p}\left(
x\right)  }{2\left(  x-1\right)  ^{2}}-2x+\frac{2e_{3}}{b}\right)
+\frac{n_{3}\left(  n_{3}+1\right)  }{2}\left(  \frac{\mathfrak{p}\left(
x\right)  }{2\left(  x-t\right)  ^{2}}-2x+\frac{2e_{2}}{b}\right) \\
+\frac{3}{4}\left(  \frac{\mathfrak{p}\left(  x\right)  +\mathfrak{p}\left(
\lambda \right)  }{2\left(  x-\lambda \right)  ^{2}}-2x-\frac{2}{b}(\wp \left(
p\right)  +e_{1})\right)  +A\left(  \frac{2}{b}\zeta \left(  p\right)
-\frac{\wp^{\prime}\left(  p\right)  }{b^{2}\left(  x-\lambda \right)
}\right)
\end{array}
\right]  .}%
\]
{Since $p$ is an apparent singularity of (\ref{89}), equation (\ref{106}) has
no logarithmic solutions at $\lambda$. The Riemann scheme of (\ref{106}) is as
follows%
\begin{equation}
\left(
\begin{array}
[c]{ccccc}%
0 & 1 & t & \infty & \lambda \\
-\frac{n_{1}}{2} & -\frac{n_{2}}{2} & -\frac{n_{3}}{2} & -\frac{n_{0}}{2} &
-\frac{1}{2}\\
\frac{n_{1}+1}{2} & \frac{n_{2}+1}{2} & \frac{n_{3}+1}{2} & \frac{n_{0}+1}{2}
& \frac{3}{2}%
\end{array}
\right)  . \label{107}%
\end{equation}
Now consider a gauge transformation $y\left(  x\right)  =\left(
x-\lambda \right)  ^{-\frac{1}{2}}x^{-\frac{n_{1}}{2}}(x-1)^{-\frac{n_{2}}{2}%
}(x-t)^{-\frac{n_{3}}{2}}\hat{y}\left(  x\right)  $. Then the Riemann scheme
for $\hat{y}\left(  x\right)  $ is
\begin{equation}
\left(
\begin{array}
[c]{ccccc}%
0 & 1 & t & \infty & \lambda \\
0 & 0 & 0 & \hat{\alpha} & 0\\
n_{1}+\frac{1}{2} & n_{2}+\frac{1}{2} & n_{3}+\frac{1}{2} & \hat{\alpha}%
+n_{0}+\frac{1}{2} & 2
\end{array}
\right)  , \label{108}%
\end{equation}
where }$\hat{\alpha}${$=-\frac{1}{2}\left(  1+n_{0}+n_{1}+n_{2}+n_{3}\right)
$. Moreover, $\hat{y}\left(  x\right)  $ satisfies the second order Fuchsian
equation%
\begin{equation}
\hat{y}^{\prime \prime}\left(  x\right)  +\hat{p}_{1}\left(  x,t\right)
\hat{y}^{\prime}\left(  x\right)  +\hat{p}_{2}\left(  x,t\right)  \hat
{y}\left(  x\right)  =0, \label{109}%
\end{equation}
where%
\begin{equation}
\hat{p}_{1}\left(  x,t\right)  =\frac{\frac{1}{2}-n_{1}}{x}+\frac{\frac{1}%
{2}-n_{2}}{x-1}+\frac{\frac{1}{2}-n_{3}}{x-t}-\frac{1}{x-\lambda} \label{110}%
\end{equation}
and{\allowdisplaybreaks%
\begin{align}
&  \hat{p}_{2}\left(  x,t\right)  =\frac{3}{4\left(  x-\lambda \right)  ^{2}%
}\label{111}\\
&  +\frac{2n_{1}n_{2}-n_{1}-n_{2}}{4x\left(  x-1\right)  }+\frac{2n_{2}%
n_{3}-n_{2}-n_{3}}{4\left(  x-1\right)  \left(  x-t\right)  }+\frac
{2n_{1}n_{3}-n_{1}-n_{3}}{4x\left(  x-t\right)  }\nonumber \\
&  +\frac{2n_{1}-1}{4x\left(  x-1\right)  }+\frac{2n_{2}-1}{4\left(
x-1\right)  \left(  x-\lambda \right)  }+\frac{2n_{3}-1}{4\left(  x-t\right)
\left(  x-\lambda \right)  }\nonumber \\
&  -\frac{1}{\mathfrak{p}\left(  x\right)  }\left[
\begin{array}
[c]{l}%
n_{0}\left(  n_{0}+1\right)  \left(  x+\frac{e_{1}}{b}\right)  -n_{1}\left(
n_{1}+1\right)  \left(  x+\frac{2e_{1}}{b}\right) \\
-n_{2}\left(  n_{2}+1\right)  \left(  x-\frac{e_{3}}{b}\right)  -n_{3}\left(
n_{3}+1\right)  \left(  x-\frac{e_{2}}{b}\right) \\
+\frac{3}{4}\left(  \frac{\mathfrak{p}\left(  x\right)  +\mathfrak{p}\left(
\lambda \right)  }{2\left(  x-\lambda \right)  ^{2}}-2x-2\frac{\wp \left(
p\right)  +e_{1}}{b}\right) \\
+A\left(  \frac{2}{b}\zeta \left(  p\right)  -\frac{\wp^{\prime}\left(
p\right)  }{b^{2}\left(  x-\lambda \right)  }\right)  +\frac{B}{b}%
\end{array}
\right]  .\nonumber
\end{align}
}Since $p\not \in E_{\tau}\left[  2\right]  $ and equation (\ref{106}) has no
logarithmic solutions at $\lambda$, it follows that $\lambda \not \in \left \{
0,1,t,\infty \right \}  $ and $\lambda$ is an apparent singularity of
(\ref{109}). Thus, $\hat{p}_{2}\left(  x,t\right)  $ can be written into the
form of (\ref{97}) with%
\begin{equation}
\hat{\kappa}=-\frac{1}{4}\left(  n_{0}-n_{1}-n_{2}-n_{3}\right)  \left(
1+n_{0}+n_{1}+n_{2}+n_{3}\right)  , \label{114}%
\end{equation}
{\allowdisplaybreaks%
\begin{align}
\mu &  =\underset{x=\lambda}{\text{ Res }}\hat{p}_{2}\left(  x,t\right)
\label{112}\\
&  =\frac{2n_{3}-1}{4\left(  \lambda-t\right)  }+\frac{2n_{2}-1}{4\left(
\lambda-1\right)  }+\frac{2n_{1}-1}{4\lambda}+\frac{3}{8}\frac{\mathfrak{p}%
^{\prime}\left(  \lambda \right)  }{\mathfrak{p}\left(  \lambda \right)  }%
+\frac{A\wp^{\prime}\left(  p\right)  }{b^{2}\mathfrak{p}\left(
\lambda \right)  },\nonumber
\end{align}
}%
\[
K=-\underset{x=t}{\text{Res}}\text{ }\hat{p}_{2}\left(  x,t\right)  =\tilde
{K},
\]
where{\allowdisplaybreaks%
\begin{align}
\tilde{K}  &  \doteqdot-\frac{2n_{2}n_{3}-n_{2}-n_{3}}{4\left(  t-1\right)
}-\frac{2n_{1}n_{3}-n_{1}-n_{3}}{4t}-\frac{2n_{3}-1}{4\left(  t-\lambda
\right)  }\label{113}\\
&  +\frac{1}{4t\left(  t-1\right)  }\left[
\begin{array}
[c]{l}%
\frac{3}{2}\frac{\lambda \left(  \lambda-1\right)  }{\left(  \lambda-t\right)
}-\frac{3}{2}\frac{\wp \left(  p\right)  +e_{3}}{b}+\frac{A\wp^{\prime}\left(
p\right)  }{\left(  \lambda-t\right)  b^{2}}\\
+\frac{n_{0}\left(  n_{0}+1\right)  e_{3}}{b}+\frac{n_{1}\left(
n_{1}+1\right)  e_{2}}{b}+\frac{n_{2}\left(  n_{2}+1\right)  e_{1}}{b}\\
-\frac{2n_{3}\left(  n_{3}+1\right)  e_{3}}{b}+\frac{2}{b}A\wp \left(
p\right)  +\frac{B}{b}%
\end{array}
\right]  .\nonumber
\end{align}
}Since{ $\lambda$ is an apparent singularity,} we conclude from (\ref{98})
that%
\begin{equation}
\tilde{K}=\frac{1}{t\left(  t-1\right)  }\left \{
\begin{array}
[c]{l}%
\lambda \left(  \lambda-1\right)  \left(  \lambda-t\right)  \mu^{2}+\hat
{\kappa}\left(  \lambda-t\right) \\
-\left[
\begin{array}
[c]{l}%
\theta_{0}\left(  \lambda-1\right)  \left(  \lambda-t\right)  +\theta
_{1}\lambda \left(  \lambda-t\right) \\
+\left(  \theta_{t}-1\right)  \lambda \left(  \lambda-1\right)
\end{array}
\right]  \mu
\end{array}
\right \}  . \label{113-0}%
\end{equation}
}

{ Conversely, for a given second order Fuchsian equation (\ref{90}) satisfying
(\ref{91}) and (\ref{137}), we know that $p_{1}(x)$, $p_{2}(x)$ and $K$ are
given by (\ref{96})-(\ref{98}), where%
\begin{equation}
\hat{\kappa}=\hat{\alpha}\left(  \hat{\alpha}+\theta_{\infty}\right)  .
\label{116}%
\end{equation}
Define }$\pm${$p$, $n_{i}$ ($i=0,1,2,3$), $A$, and $B$ by
\begin{equation}
\lambda=\frac{\wp \left(  p\right)  -e_{1}}{e_{2}-e_{1}}, \label{117}%
\end{equation}%
\begin{equation}
\left(  \theta_{0},\theta_{1},\theta_{t},\theta_{\infty}\right)  =\left(
n_{1}+\frac{1}{2},n_{2}+\frac{1}{2},n_{3}+\frac{1}{2},n_{0}+\frac{1}%
{2}\right)  , \label{115}%
\end{equation}%
\begin{equation}
\mu=\frac{2n_{3}-1}{4\left(  \lambda-t\right)  }+\frac{2n_{2}-1}{4\left(
\lambda-1\right)  }+\frac{2n_{1}-1}{4\lambda}+\frac{3}{8}\frac{\mathfrak{p}%
^{\prime}\left(  \lambda \right)  }{\mathfrak{p}\left(  \lambda \right)  }%
+\frac{A\wp^{\prime}\left(  p\right)  }{b^{2}\mathfrak{p}\left(
\lambda \right)  }, \label{118}%
\end{equation}
and%
\begin{equation}
B=A^{2}-\zeta \left(  2p\right)  A-\frac{3}{4}\wp \left(  2p\right)  -\sum
_{i=0}^{3}n_{i}\left(  n_{i}+1\right)  \wp \left(  p+\frac{\omega_{i}}%
{2}\right)  . \label{119}%
\end{equation}
Since $\lambda \not \in \{0,1,t,\infty \}$, $p\not \in E_{\tau}\left[  2\right]
$. Thus $\wp^{\prime}\left(  p\right)  \not =0$ and $A$ is well-defined by
(\ref{118}). In order to obtain the corresponding generalized Lam\'{e}
equation (\ref{89}), it suffices to prove that $p_{1}\left(  x,t\right)  $ and
$p_{2}\left(  x,t\right)  $ can be expressed in the form of (\ref{110}) and
(\ref{111}). By (\ref{96}) and (\ref{115}), {it is easy to see that}
$p_{1}\left(  x,t\right)  $ is of the form (\ref{110}). By (\ref{113}) and
(\ref{113-0}), we see that $K$ can be written into (\ref{113}), so
$p_{2}\left(  x,t\right)  $ can also be expressed in the form of (\ref{111}).
Finally, the assertion that }$p$ is an apparent singularity follows from the
assumption that {$\lambda$ is an apparent singularity of (\ref{90}) (or
follows from (\ref{119}) and Lemma \ref{lem-apparent}). This completes the
proof.}
\end{proof}


\begin{thebibliography}{99}                                                                                               %


\bibitem {Babich-Bordag}{ M. V. Babich and L. A. Bordag; The elliptic form of
the sixth Painlev\'{e} equation. Preprint NT Z25/1997, Leipzig (1997). }

\bibitem {YB}{ Y. V. Brezhnev; Non-canonical extension of $\vartheta
$-functions and modular integrability of $\vartheta$-constants. Proc. Roy.
Soc. Edinburgh Sect. A 143 (2013), no. 4, 689--738. }

\bibitem {Chai-Lin-Wang}C.L. Chai, C.S. Lin and C.L. Wang; Mean field
equations, hyperelliptic curves, and modular forms: I, Cambridge Journal of
Mathematics, 3 (2015), 127-274.

\bibitem {Chen-Kuo-Lin}Z. Chen, T.J. Kuo and C.S. Lin; Mean field equation,
Isomonodromic deformation and Painlev\'{e} VI equation: Part I. preprint 2015.

\bibitem {DIKZ}P. Deift, A. Its, A. Kapaev and X. Zhou; On the
algebro-Geometric integration of the Schlesinger equations. Comm. Math. Phys.
203 (1999), 613-633.{ }

\bibitem {Dubrovin-Mazzocco}{ B. Dubrovin and M. Mazzocco; Monodromy of
certain Painlev\'{e}-VI transcendents and reflection groups. Invent. Math. 141
(2000), 55-147. }

\bibitem {Fuchs}R. Fuchs; \"{U}ber lineare homogene Differentialgleichungen
zweiter Ordnung mit drei im Endlichen gelegenen wesentlich singul\"{a}ren
Stellen, Math. Ann. 63 (1907), 301-321.

\bibitem {GR}{ V. Gromak, I. Laine and S. Shimomura; Painlev\'{e} differential
equations in the complex plane. de Gruyter Studies in Mathematics 28, Berlin.
New York 2002. }

\bibitem {Guzzetti}{ D. Guzzetti; The elliptic representation of the general
Painlev\'{e} VI equation. Comm. Pure Appl. Math. 55 (2002), 1280-1363. }

\bibitem {Halphen}{ G.H. Halphen; Trait\'{e} des Fonctions Elliptique II,
1888.}

\bibitem {Hit1}{ N. J. Hitchin; Twistor spaces, Einstein metrics and
isomonodromic deformations. J. Differ. Geom. 42 (1995), no.1, 30-112. }

\bibitem {Hit2}{ N. J. Hitchin; Poncelet polygons and the Painlev\'{e}
transcendents. Geometry and analysis, Tata Inst. Fund. Res., Bombay, 151--185,
1995.}

\bibitem {Inaba-Iwasaki-Saito}M. Inaba, K. Iwasaki and M. Saito; B\"{a}cklund
transformations of the sixth Painlev\'{e} equation in terms of Riemann-Hilbert
correspodence. Inter. Math. Res. Not. 1 (2004), 1-30.

\bibitem {IW}{K. Iwasaki, }Moduli and deformation for Fuchsian projective
connections on a Riemann surface. J. Fac. Sci. Univ. Tokyo Sect. IA Math. 38
(1991), no. 3, 431--531.

\bibitem {GP}{ K. Iwasaki, H. Kimura, S. Shimomura and M. Yoshida; From Gauss
to Painlev\'{e}: A Modern Theory of Special Functions. Springer vol. E16,
1991.}

\bibitem {Jimbo-Miwa}M. Jimbo, T. Miwa and K. Ueno; Monodromy perserving
deformation of linear ordinary differential equations with rational
coefficients. I. Phys. 2D. 2 (1981), 306--352.

\bibitem {Kawai}S. Kawai; Isomonodromic deformation of Fuchsian projecitve
connections on elliptic curves. Nagoya Math. J. 171 (2003), 127-161.

\bibitem {Lisovyy-Tykhyy}{ O. Lisovyy and Y. Tykhyy; Algebraic solutions of
the sixth Painlev\'{e} equation. J. Geom. Phys. 85 (2014), 124-163. }

\bibitem {Y.Manin}{ Y. Manin; Sixth Painlev\'{e} quation, universal elliptic
curve, and mirror of $\mathbb{P}^{2}$. Amer. Math. Soc. Transl. (2), 186
(1998), 131--151. }

\bibitem {Okamoto2}{ K. Okamoto; Isomonodromic deformation and Painlev\'{e}
equations, and the Garnier system. J. Fac. Sci.\ Univ. Tokyo Sec. IA Math. 33
(1986), 575-618. }

\bibitem {Okamoto1}{ K. Okamoto; Studies on the Painlev\'{e} equations. I.
Sixth Painlev\'{e} equation $P_{VI}$. Ann. Mat. Pura Appl. 146 (1986),
337-381. }

\bibitem {Okamoto}{ K. Okamoto; On the holonomic deformation of linear
ordinary differential equations on an elliptic curve. Kyushu J. Math. 49
(1995), 281-308. }

\bibitem {Painleve}{ P. Painlev\'{e}; Sur les \'{e}quations
diff\'{e}rentialles du second ordre \`{a} points critiques fixes. C. R.
Acad.\ Sic. Paris S\'{e}r. I 143 (1906), 1111-1117. }

\bibitem {Poole}{ E. Poole; Introduction to the theory of linear differential
equations. Oxford University Press, 1936.}

\bibitem {Sasaki}Y. Sasaki; On the gauge transformation of the sixth
Painlev\'{e} equation. Painlev\'{e} equations and related topics, 137--150,
Walter de Gruyter, Berlin, 2012.

\bibitem {Whittaker-Watson}{ E. Whittaker and G. Watson, A course of modern
analysis. Cambridge University Press, 1996 }
\end{thebibliography}
\end{document}